\documentclass[11pt]{article}

\usepackage{authblk}
\usepackage{fullpage}  % for narrow margins
\usepackage{float}
\usepackage{pb-diagram}  		% for commutative diagrams
\usepackage{esvect}
\usepackage[utf8]{inputenc} % allow utf-8 input
\usepackage[T1]{fontenc}    % use 8-bit T1 fonts
\usepackage{hyperref}       % hyperlinks
\usepackage{url}            % simple URL typesetting
\usepackage{booktabs}       % professional-quality tables
\usepackage{amsfonts}       % blackboard math symbols
\usepackage{nicefrac}       % compact symbols for 1/2, etc.
\usepackage{microtype}      % microtypography
\usepackage{fullpage}  % for narrow margins
\usepackage{float}
\usepackage{pb-diagram}  
\usepackage{makecell}		% for commutative diagrams
\usepackage{url}

\usepackage{algorithm, algorithmicx,algpseudocode, listings} %format of the algorithm
\usepackage{multirow} %//multirow for format of table
\usepackage{hhline}  % for \hhline in tables
\usepackage{amsmath}
\usepackage{xcolor}

\usepackage{graphicx}
\usepackage{amscd}
\usepackage{amssymb,mathrsfs}
\input{epsf.sty}
\usepackage{amsthm,amscd}
\usepackage{color}
\usepackage{latexsym}
\usepackage{epic}
\usepackage{appendix}
\usepackage{enumerate}
\usepackage{longtable}
\usepackage{lscape}
\usepackage{extarrows}
\usepackage{epstopdf}
\usepackage{listings}
\newlength\myindent

\renewcommand{\theequation}{\arabic{equation}}

\newcommand{\Rmnum}[1]{\expandafter\@slowromancap\romannumeral #1@}
\newcommand{\inner}[3][]{{\langle #2,#3 \rangle_{#1}}}

\newcommand{\etal}{\textit{et al. }}
\allowdisplaybreaks
\newcommand\numberthis{\addtocounter{equation}{1}\tag{\theequation}}

\newcommand{\whcomm}[2]{{}{#2}}
\newcommand{\whcommsec}[2]{{}{ #2}}
\newcommand{\whcommtrd}[2]{{}{#2}}
\newcommand{\kwcomm}[2]{{}{#2}}
\newcommand{\kw}[1]{{#1}}

\newcommand{\whfirrev}[2]{{}{#2}}

\newcommand{\whsecrev}[2]{{}{#2}}
\newtheorem{definition}{Definition}[section]
\newtheorem{theorem}{Theorem}[section]
\newtheorem{lemma}{Lemma}[section]

\newtheorem{assumption}{Assumption}[section]

\numberwithin{equation}{section}

\renewcommand{\P}{\mathrm{P}}
\DeclareMathOperator{\T}{\mathrm{T}}
\DeclareMathOperator{\Hess}{\mathrm{Hess}}
\DeclareMathOperator{\grad}{\mathrm{grad}}

\DeclareMathOperator{\Prox}{\mathrm{Prox}}

\DeclareMathOperator{\N}{\mathrm{N}}
\DeclareMathOperator{\D}{\mathrm{D}}

\DeclareMathOperator{\trace}{\mathrm{trace}}

\DeclareMathOperator{\diag}{\mathrm{diag}}
\DeclareMathOperator{\St}{\mathrm{St}}

\DeclareMathOperator*{\argmin}{arg\,min}
\def\S{\mathbb{S}}
\def\R{\mathbb{R}}
\def\M{\mathcal{M}}
\begin{document}

\title{An Extension of Fast Iterative Shrinkage-thresholding to Riemannian Optimization for Sparse Principal Component Analysis\footnotetext{Authors  are listed alphabetically, and corresponding authors: Wen Huang (\texttt{wen.huang@xmu.edu.cn})  and Ke Wei (\texttt{kewei@fudan.edu.cn}). WH was partially  supported by the Fundamental Research Funds for the Central Universities (NO. 20720190060) and National Natural Science Foundation of China (NO. 12001455). KW was partially  supported by the NSFC Grant 11801088 and the Shanghai Sailing Program 18YF1401600.}}
\author[1]{Wen Huang}
\author[2]{Ke Wei}

\affil[1]{ School of Mathematical Sciences, Xiamen University, Xiamen, China.\vspace{.15cm}}
\affil[2]{School of Data Science, Fudan University, Shanghai, China.}

\maketitle

%\setcounter{tocdepth}{2}
%\tableofcontents

\begin{abstract}
{
Sparse principal component analysis (PCA), an important variant of PCA,  attempts to find sparse loading vectors when conducting dimension reduction. This paper considers the \kw{nonsmooth} Riemannian optimization problem associated with the ScoTLASS model \cite{JoTrUd2003a} for sparse PCA which can impose orthogonality and sparsity simultaneously. 
\kw{A Riemannian proximal method is proposed in the work of Chen \etal \cite{CMSZ2019} for the efficient solution of this optimization problem. In this paper, two acceleration schemes are introduced. First  and foremost, we extend the  FISTA  method  from the Euclidean space to the Riemannian manifold to solve sparse PCA, leading to the accelerated Riemannian proximal gradient method. Since the Riemannian optimization problem for sparse PCA  is essentially non-convex, a restarting technique is  adopted to stabilize the  accelerated  method without sacrificing the fast convergence. Second, a diagonal preconditioner is proposed for  the  Riemannian proximal subproblem which can further accelerate the convergence of the Riemannian proximal methods.
Numerical evaluations establish the computational advantages of the proposed methods over the existing proximal gradient methods on a manifold.
Additionally,  a short result concerning the convergence of the Riemannian subgradients of a sequence is established, which, together with the result in the work of Chen \etal \cite{CMSZ2019}, can show the stationary point convergence of the Riemannian proximal methods.
}
}
\end{abstract}
%%%%%%%%%%%

\section{Introduction}\label{sec:int}

Principal component analysis (PCA) is an important data processing technique.
In essence, PCA attempts to find a low dimensional representation of a data set. The low dimensional representation can be subsequently used for data denoising, vision and recognition, just to name a few. However, due to the complexity of data as well as the interpretability issues, vanilla PCA may not be able to meet the requirements of real applications. Therefore,  several variants of PCA have been proposed and studied, one of which is sparse PCA.

Given a dataset, PCA aims to find  linear combinations of the original variables such that the new variables can capture the maximal variance in the data. In order to achieve the maximal variance, PCA tends to use a linear combination of all the variables. Thus, all coefficients (loadings) in the linear combination are typically non-zero, which will cause interpretability issues in many applications. For example, in genome data analysis, each coefficient may correspond to a specific gene, and it is more desirable to have the new variable being composed of only a few genes. This means that the loading vector should have very few non-zero entries.

Let $A$ be an $m\times n$ data matrix, where $m$  denotes the number of samples and $n$ denotes the number of variables. Without loss of generality, assume each column of $A$ has zero mean. Then PCA can be formally expressed as the following maximization problem:
\begin{align*}
\max_{X\in\R^{n\times p}} \|AX\|_{\mathrm{F}}^2 \quad\mbox{subject to}\quad X^T X = I_p,
\numberthis\label{eq:PCA}
\end{align*}
where each column $X$ denotes a loading vector. The PCA problem admits a closed form solution which can be computed via the singular value decomposition (SVD) of the data matrix.  However, it seldom yields a sparse solution; that is, each column of $X$ is very likely to be a dense vector. Alternatively,  sparse PCA attempts to achieve a better trade-off between the variance of $AX$ and the sparsity of $X$. In this paper we consider the following model for  sparse PCA:
\begin{align*}
\min_{X\in\R^{n\times p}}-\|AX\|_{\mathrm{F}}^2+\lambda\|X\|_1 \quad\mbox{subject to}\quad X^T X = I_p,\numberthis\label{eq:sparsePCA}
\end{align*}
where $\|X\|_1= \sum_{i,j}|X_{ij}|$ imposes the sparsity of $X$   and $\lambda>0$ is a tuning parameter controlling the balance between variance and sparsity.

In fact, \eqref{eq:sparsePCA} is a penalized version of the ScoTLASS model proposed by Jolliffe \etal \cite{JoTrUd2003a}, which is inspired by the Lasso regression. In addition to the ScoTLASS model, there are many other formulations for sparse PCA. By rewriting  PCA as a regression optimization problem, Zou \etal \cite{ZoHaTi2006a} propose a model which mixes the ridge regression  and the Lasso regression. A semidefinite programming is proposed in the work of d'Aspremont \etal \cite{AsBaGh2008a,AsGhJoLa2007a} to compute the dominant sparse loading vector. In the work of Shen and Huang\cite{ShHu2008a} and Witten \etal \cite{WiTiHa2009a}, sparse PCA is studied based on matrix decompositions. A formulation similar to \eqref{eq:sparsePCA} but with decoupled variables is investigated in \cite{JoNeRiSe2010a}. Moreover, different algorithms have been developed for different formulations. We refer interested readers to \cite{ZouXue2018a} for a nice overview of sparse PCA on both computational and theoretical results.

%\paragraph{Main Contribution} Due to the simultaneous existence of the orthogonal constraint and the non-smooth term in \eqref{eq:sparsePCA}, it is quite challenging to develop fast algorithms to compute its solution. In \cite{CMSZ2019}, Chen \etal proposed a Riemannian proximal gradient  method  call ManPG for \eqref{eq:sparsePCA}. In this paper  we extend the fast iterative shrinkage-thresholding algorithm (FISTA, \cite{BecTeb2009fista}) to solve \eqref{eq:sparsePCA}.  Empirical comparisons clearly show the accelerated property of the algorithm over the Riemannian proximal gradient method, as in the Euclidean case. In addition, convergence of the algorithm to stationary points is also properly justified.

Due to the simultaneous existence of the orthogonal constraint and the non-smooth term in \eqref{eq:sparsePCA}, it is quite challenging to develop fast algorithms to compute its solution. In the work of Chen \etal \cite{CMSZ2019},  a Riemannian proximal gradient  method  called ManPG is proposed for this problem. In this paper  we extend the fast iterative shrinkage-thresholding algorithm (FISTA\cite{BecTeb2009fista}) to solve \eqref{eq:sparsePCA}.   For ease of exposition, we consider the following more general nonconvex optimization problem:\footnote{It is often more convenient to use lowercase letters to denote matrices  when presenting the  problem, the algorithms as well as the theoretical results.}
\begin{equation*}
\min F(x) = f(x) + g(x)\quad\mbox{subject to}\quad x\in\M,\numberthis\label{eq:optimization}
\end{equation*}
where $\mathcal{M} \subset \mathbb{R}^{n \times m}$ is a compact Riemannian submanifold, $f: \mathbb{R}^{n \times m} \rightarrow \mathbb{R}$ is $L$-continuously differentiable (may be nonconvex) and $g$ is continuous, convex, but may be nondifferentiable. Clearly, \eqref{eq:sparsePCA} is a special case of \eqref{eq:optimization} with $\mathcal{M}$ being the Stiefel manifold, defined by 
\begin{align*}
\St(p,n) = \{X\in\R^{n\times p}~|~X^TX=I_p\}.\numberthis\label{eq:Stiefel}
\end{align*}

When $F$ is a smooth function (i.e., $g=0$), most of the standard optimization algorithms for the Euclidean setting, for example the (accelerated) gradient method, the Newton method and the BFGS method, and the trust region method, are readily extended to the Riemannian setting; see  the work\cite{AMS2008,HUANG2013,Bart2010,Boumal2014,Mishra2014} and references therein.  

There have also been many algorithms that are designed for the nonsmooth optimization problems on manifold. 
In the work of Ferreira and Oliveira\cite{FerOli1998}, a subgradient method is studied for minimizing a convex function on a Riemannian manifold and convergence guarantee is established for the diminishing stepsizes.
In the paper~\cite{ZS2016}, Zhang and Sra analyze a Riemannian subgradient-based  method and show that the cost function decreases to the optimal value at the rate of $O(1/\sqrt{k})$.
 When the cost function is Lipschitz continuous, the $\epsilon$-subgradient method is a variant of the subgradient method which utilizes the gradient at nearby points as an approximation of the subgradient at a given point.
In the papers~\cite{GH2015a,GH2015b}, Grohs and Hosseini develop two $\epsilon$-subgradient-based optimization methods using line search strategy and trust region strategy, respectively. The 
convergence of the algorithms to critical points is established in their work.
Huang~\cite{HUANG2013} generalizes a gradient sampling method to the Riemannian setting, which is very  efficient for small-scale problems but lacks convergence analysis. In the paper~\cite{HU2017}, Hosseini and Uschmajew present a Riemannian gradient sampling method with convergence analysis. Recently, Hosseini \etal \cite{HHY2018} propose a new Riemannian line search  method by combining the $\epsilon$-subgradient method and the quasi-Newton ideas. The proximal point method has also been extended to the Riemannian setting. For instance, Ferreira and Oliveira propose a Riemannian proximal point method~\cite{FO2002}. The $O(1/k)$ convergence rate of the method for the Hadamard manifold is established by Bento \etal~\cite{BFM2017}. The shortcoming of the Riemannian proximal point method is that there do not exist efficient algorithms for the subproblems.

While some of the aforementioned algorithms are also applicable for the nonsmooth optimization problem \eqref{eq:optimization}, they lack the ability of exploiting the decomposable structure of the cost function. In contrast, a %manifold 
proximal gradient method  (ManPG) is proposed in the work of Chen \etal \cite{CMSZ2019} \whfirrev{}{when $\mathcal{M}$ is the Stiefel manifold in \eqref{eq:optimization}}, % \whfirrev{}{and a proximal gradient method for general Riemannian optimization is given in the work~\cite{HuaWei2019b}. These methods use different generalizations of the proximal mapping and therefore may perform differently in applications. However, they
%are both} analogues of the proximal gradient method in the Euclidean setting and hence \whfirrev{}{are} able to take advantage of the problem structure.  %\whfirrev{}{Moreover, proximal gradient methods for general Riemannian optimization are recently proposed in the work~\cite{HuaWei2019b}, which . }
which is an analogue of the proximal gradient method in the Euclidean setting and hence is able to take advantage of the problem structure. \whfirrev{}{Moreover, Riemannian proximal methods for composite problems on general manifolds are developed in the work of the same authors \cite{HuaWei2019b} based on a different Riemannian proximal mapping. As suggested in that work, for optimization problems on Stiefel manifold (the focus of this paper), the solution to  the Riemannian proximal mapping used in the work of Chen \etal~\cite{CMSZ2019} and this paper can be solved more efficiently. }

%\whcommtrd{}{Add the theoretical contribution?}
\kw{\em The main contributions of this paper are summarized as follows.}
 We first extend the accelerated proximal gradient method (specifically FISTA \cite{BecTeb2009fista}) to the Riemannian setting to solve \eqref{eq:optimization}. The algorithm is coined as AManPG (accelerated ManPG ). A simple safeguard is introduced in AManPG so that its convergence to stationary points can be guaranteed.
 Empirical comparisons clearly show that as in the Euclidean case AManPG exhibits a faster convergence rate than ManPG.  Moreover, a weighted proximal subproblem is considered  in this paper and we observe that a  computationally efficient weight in the diagonal form can further speed up the Riemannian proximal gradient methods. 
\kw{It has been shown in the work \cite{CMSZ2019}  that the search direction computed in ManPG converges to zero. In this paper a complementary result about the convergence of the Riemannian subgradients of a sequence is provided, which can be used to complete the stationary point analysis of the Riemannian proximal methods together with the  result in the work \cite{CMSZ2019}.} 

The remainder of this paper is organized as follows. In Section~\ref{sec:pre}, we give some basic facts about Riemannian manifolds and Riemannian optimization. The accelerated Riemannian proximal gradient method (i.e., AManPG)
is presented in Section~\ref{sec:alg} together with the preliminary convergence analysis. Empirical performance evaluations are presented in \ref{sec:num}, while Section~\ref{sec:con} concludes this paper with a few future directions.
%%%%%%%%%%%
%%%%%%%%%%%
%\subsubsection*{Acknowledgments}

\section{Preliminaries on Manifold}\label{sec:pre}
%The optimization problem \eqref{eq:sparsePCA} fits in the general framework of Riemannian optimization as the set of $n\times p$ orthogonal matrices forms a smooth manifold. This manifold is known as the Stiefel manifold, typically denoted $\St(p,n)$,
%\begin{align*}
%\St(p,n) = \{X\in\R^{n\times p}~|~X^TX=I_p\}.%\numberthis\label{eq:Stiefel}
%\end{align*}
This section reviews some basic notation on Riemannian manifold that is closely related to the work in this paper. We focus on submanifolds of Euclidean spaces with $\St(p,n)$ as an example since in this case the manifold is geometrically more intuitive and can be  imagined as a smooth surface in a 3D space. Interested readers are referred to the book\cite{AMS2008} for more details about Riemannian manifolds and Riemannian optimization.

Assume $\M$ is a smooth submanifold of a Euclidean space and let $x\in\M$. The tangent space of $\M$ at $x$, denoted $\T_x\M$, is a collection of  derivatives  of all the smooth curves passing through $x$,
\begin{align*}
\T_x\M=\{\gamma'(0)~|~\gamma(t)\mbox{ is a curve in }\M\mbox{ with }\gamma(0)=x\}.
\end{align*}
The tangent space is a vector space and each tangent vector in $\T_x\M$ corresponds to a linear mapping from the set of smooth real-valued functions in a neibourghood of $x$ to $\R$. Indeed, it is the latter property that is adopted to define  tangent spaces for abstract manifolds. Since $\T_x\M$ is a vector space, we can equip it with an inner product (or metric) $g_x(\cdot,\cdot):\T_x\M\times \T_x\M\rightarrow\R$; see Figure~\ref{figure5} (left) for an illustration. A manifold whose tangent spaces are endowed with a smoothly varying metric is referred to as a Riemannian manifold. For  a smooth function $f$ defined a Riemannian manifold, the Riemannian gradient of $f$ at $x$, denoted $\grad f(x)$, is the unique tangent vector such that $g_x(\grad f(x),\eta_x)=\D f(x)[\eta_x],~\forall \eta_x\in \T_x\M$, where $\D f(x)[\eta_x]$ is the directional derivative of $f$ along the direction $\eta_x$.
Moreover, the Riemannian gradient of  $f$ at $x$ is simply the orthogonal projection of $\nabla f(x)$ onto $\T_x\M$; that is, %$\grad f(x)=\P_{\T_x\M}\nabla f(x)$.% when the Euclidean metric is used.
\begin{align*}
\grad f(x)=\P_{\T_x\M}\nabla f(x),\numberthis\label{eq:grad}
\end{align*}
where $\nabla f(x)$ is the Euclidean gradient of $f$ at $x$.

When we construct the diagonal weight for the proximal subproblem in the algorithm, the second order information of a function on the Riemannian manifold will also be needed. The Riemannian Hessian of $f$ at $x$, denoted $\Hess f(x)$, is a mapping from $\T_x\M$ to $\T_x\M$. Moreover, when $\M$ is a Riemannian submanifold of a Euclidean space $\Hess f(x)$ satisfies 
\begin{align*}
\Hess f(x)[\eta_x] =\P_{\T_x\M}\D\grad f(x)[\eta_x],\quad \eta_x\in\T_x\M,\numberthis\label{eq:hessian}
\end{align*} 
where $\D\grad f(x)[\eta_x]$ denotes the directional derivative of $\grad f(x)$ along the direction $\eta_x$.
%%%%%%%

\begin{figure}
\centering
\scalebox{0.5}{
  \setlength{\unitlength}{1bp}%
  \begin{picture}(356.48, 187.96)(0,0)
  \put(0,0){\includegraphics{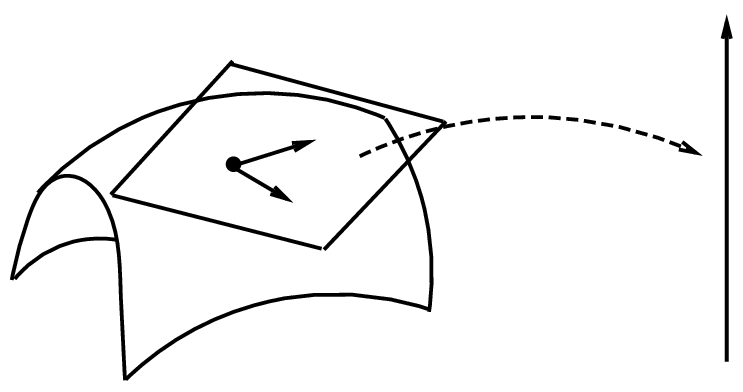}}
  \put(185.88,53.02){\fontsize{14.23}{17.07}\selectfont $\mathcal{M}$}
  \put(96,95){\fontsize{14.23}{17.07}\selectfont $x$}
  \put(146.20,125.96){\fontsize{14.23}{17.07}\selectfont $\xi_x$}
  \put(143.22,88.85){\fontsize{14.23}{17.07}\selectfont $\eta_x$}
  \put(331.34,171.18){\fontsize{14.23}{17.07}\selectfont $\mathbb{R}$}
  \put(231.59,140.89){\fontsize{14.23}{17.07}\selectfont $g_x(\eta_x,\xi_x)$}
  \put(118.55,161.37){\fontsize{14.23}{17.07}\selectfont $\mathrm{T}_x \mathcal{M}$}
  \end{picture}%
  }
  \;\;\;\;\;\;\;\;\;\;\;\;  \;\;\;\;\;\;\;\;\;\;\;\;
\scalebox{0.5}{
  \setlength{\unitlength}{1bp}%
  \begin{picture}(219.70, 178.15)(0,0)
  \put(0,0){\includegraphics{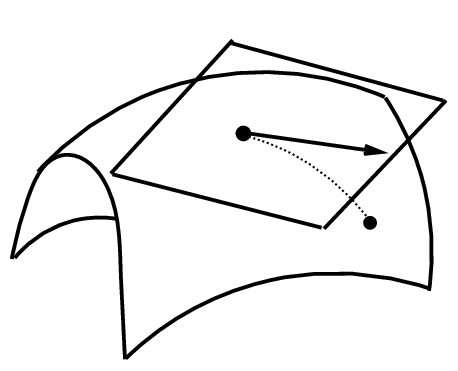}}
  \put(74.18,33.70){\fontsize{14.23}{17.07}\selectfont $\mathcal{M}$}
  \put(102.98,103.21){\fontsize{14.23}{17.07}\selectfont $x$}
  \put(169.61,112.5){\fontsize{14.23}{17.07}\selectfont $\eta_x$}
  \put(118.55,161.37){\fontsize{14.23}{17.07}\selectfont $\mathrm{T}_x \mathcal{M}$}
  \put(163.50,56.51){\fontsize{14.23}{17.07}\selectfont $R_x(\eta_x)$}
  \end{picture}%
  }
  \caption{(Left) Riemannian metric; (Right) Retraction.}
  \label{figure5}
\end{figure}
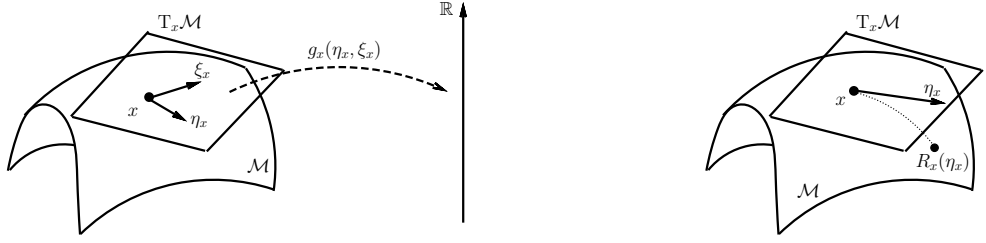

Regarding the Stiefel manifold, the tangent space of $\St(p,n)$ at a matrix $X$ is given by 
\begin{align*}
\T_X\St(p,n) =\{\eta_X\in\R^{n\times p}~|~X^T\eta_X+\eta_X^TX=0\}.\numberthis\label{eq:tangent}
\end{align*}
In particular, when $p=1$, $\St(p,n)$ is the unit sphere $\S^{n-1}$ in $\R^n$ and $\T_x\S^{n-1}$ consists of those vectors that are perpendicular to $x$. We can use the inner product inherited from $\R^{n\times p}$ as the Riemannian metric on $\T_X\St(p,n)$; that is,
\begin{align*}
g_X(\xi_X,\eta_X) = \trace(\xi_X^T\eta_X),\quad\forall~ \xi_X,\eta_X\in \T_X\St(p,n).
\end{align*}
Under this metric, the projection of any $n\times p$ matrix $\xi$ onto $\T_X\St(p,n)$
is given by
\begin{align*}
\P_{\T_X\St(p,n)} = \xi - X\mbox{sym}(X^T\xi),\quad\mbox{where }\mbox{sym}(X^T\xi)=\frac{X^T\xi+\xi^TX}{2}.\numberthis\label{eq:proj}
\end{align*}

A Riemannian optimization algorithm typically conducts a line search or solves  a linear system or a model problem on a tangent space, and then moves the solution back to the manifold.  The notion of retraction plays a key role in mapping vectors in a tangent space to points on a manifold.
\begin{definition}[Retraction]\label{def:retraction} At $x\in\M$, a retraction $R_x(\cdot)$ is a smooth mapping from $\T_x\M$ to $\M$ which satisfies the following two properties:
1) $R_x(0_x)=x$, where $0_x$ is the zero element in $\T_x\M$;
2) $\left.\frac{d}{dt}R_x(t\eta_x)\right|_{t=0} = \xi_x$ for any $\xi_x\in\T_x\M$.
\end{definition}
The second property  means the velocity of the curve defined by $R_x(t\eta_x)$ is equal to $\eta_x$ at $t=0$; see Figure~\ref{figure5} (right). Roughly speaking, retraction plays the role of line search when designing a Riemannian optimization algorithm; namely,
\begin{align*}
\mbox{(Euclidean) }x_{k+1} = x_k+ \eta_{x_k}\quad\Rightarrow\quad \mbox{(Riemannian) }x_{k+1} = R_{x_k}(\eta_{x_k}).\numberthis\label{eq:substitution}
\end{align*}
Note that the  two properties in Definition~\ref{def:retraction} cannot uniquely determine a retraction.  For the Stiefel manifold, several retractions can be constructed, for example those based on the exponential map, the QR factorization, the singular value decomposition (SVD) or the  polar decomposition \cite{AMS2008}. In this paper we use the one based on the SVD:
\begin{align*}
R_X(\eta_X) = UV^T,\quad\mbox{where }X+\eta_X = U\Sigma V^T\mbox{ is the SVD of } X+\eta_X.
\end{align*}
%This retraction is indeed a projective retraction CITE, i.e.,
%$$
%UV^T = \min\limits_{Z\in\St(p,n)}\|X+\eta_X-Z\|_{\mathrm{F}}.
%$$
Noticing that $X\in\St(p,n)$ and $\eta_X\in\T_X\St(p,n)$, thus $X+\eta_X$ is a matrix of full column rank. Then it is not hard to verify that the retraction based on the SVD is equivalent to the retraction based on the polar decomposition given by
%which can be computed efficiently as follows (note that both  $X$ and $\xi_X$ are $n\times p$ tall matrices):
\begin{align*}
R_X(\eta_X) = (X+\eta_X)(I_p+\eta_X^T\eta_X)^{-1/2}.\numberthis\label{eq:polar}
\end{align*}
Since $X+\eta_X$ is a tall matrix, \kwcomm{}{an alternative way to compute $R_X(\eta_X)$ is as follows:}
\begin{align*}
[Q,R] = \mbox{qr}(X+\eta_X),~[\tilde{U},\tilde{S},\tilde{V}] = \mbox{svd}(R),~R_X(\eta_X)=Q(\tilde{U}\tilde{V}^T),\numberthis\label{eq:computeRetraction}
\end{align*}
where $\mbox{qr}$ and $\mbox{svd}$ means computing the compact QR decomposition and SVD of a matrix, respectively. \whcommsec{[In C++, I use $[U, S, V] = svd(X + \eta_X)$ and $R_X(\eta_X) = U V^T$ since there is basically no time difference in C++ implementation.]}{} \kwcomm{[Done by rephrasing the statement a little bit.]}{}

When the cost function $F$ of  the Riemannian optimization problem is smooth, the first order optimality condition is 
$$
\grad F(x) =0.
$$
If $F$ is not differentiable but Lipschitz continuous, then the Riemannian version of generalized Clarke subdifferential introduced in Hosseini \etal \cite{HoPo2011,HHY2018} is used.  %the Riemannian version of generalized subdifferential defined in~\cite{HHY2018} is used \kw{same as the one used in Ma's paper?}. 
Specifically, since $\hat{F}_x = F \circ R_x$ is a Lipschitz continuous function defined on a Hilbert space $\T_x \mathcal{M}$, the  generalized  Clarke directional derivative at $\eta_x \in \T_x \mathcal{M}$, denoted by $\hat{F}_x^\circ(\eta_x; v)$, is defined by $$\hat{F}_x^\circ(\eta_x; v) = \lim_{\xi_x \rightarrow \eta_x} \sup_{t \downarrow 0} \frac{\hat{F}_x(\xi_x + t v) - \hat{F}_x(\xi_x)}{t},$$ where $v \in \T_x \mathcal{M}$. The generalized \whfirrev{}{Clarke} subdifferential of $\hat{F}_x$ at $\eta_x$, denoted by $\partial \hat{F}_x(\eta_x)$, is defined by $$\partial \hat{F}_x(\eta_x) = \{\eta_x \in \T_x \mathcal{M} \mid \inner[x]{\eta_x}{v} \leq \hat{F}_x^\circ(\eta_x; v) \hbox{ for all } v \in \T_x \mathcal{M}\}.$$  %The Riemannian version of the Clarke generalized direction derivative of $F$ at $x$ in the direction $\eta_x \in \T_x \mathcal{M}$, denoted  $F^\circ (x; \eta_x)$, is defined by $$F^\circ (x; \eta_x) = \hat{F}_x^\circ (0_x; \eta_x).$$
 The Riemannian version of generalized Clarke subdifferential of $F$ at $x$, denoted  $\hat{\partial} F(x)$, is defined as $\hat{\partial} F(x) = \partial \hat{F}_x(0_x)$. Any tangent vector $\xi_x \in \hat{\partial} F(x)$ is called a subgradient of $F$ at $x$.
For the cost function  in \eqref{eq:optimization} (or a class of regular functions in general), the generalized Clarke subdifferential is given by  \cite{YaZhSo2014} 
$$\hat{\partial} F(x)=\P_{\T_X\M}\partial F(x),$$
where $\partial F(x)$ denotes the subdifferential in the Euclidean space. Moreover, the first order optimality of the problem \eqref{eq:optimization} is given by $$0\in\hat{\partial}F(x)=\grad f(x)+\P_{\T_X\M}\partial g(x).$$
We refer the reader to the work of Yang \etal \cite{YaZhSo2014} for more details.

%%%%
\section{Extending FISTA to  Riemannian Optimization}\label{sec:alg}

%The proposed algorithm is stated in Algorithm~\ref{RFISTA}.

Before presenting the algorithm for \eqref{eq:optimization}, let us first briefly review the proximal gradient method  and accelerated proximal gradient method  for the optimization problem similar to \eqref{eq:optimization} but with the manifold constraint $x\in\M$ being dropped. In each iteration, the proximal gradient  method updates the estimate of the minimizer via\footnote{Here we write the subproblem in terms of the search direction for ease of extension to the manifold situation, but the update rule is the same as  $x_{k+1}=\argmin_x\langle\nabla f(x_k),x-x_k\rangle+\frac{1}{2\mu}\|x-x_k\|_{\mathrm{F}}^2+g(x)$.}
\begin{align*}
\begin{cases}\eta_{x_k} = \argmin_{\eta\in\R^{n\times p}}\langle\nabla f(x_k),\eta\rangle+\frac{1}{2\mu}\|\eta\|_{\mathrm{F}}^2+g(x_k+\eta)\\
x_{k+1}=x_k+\eta_{x_k},
\end{cases}\numberthis\label{eq:PG}
\end{align*}
where $\|\eta\|_{\mathrm{F}}$ denotes the Frobenius norm of $\eta$.
\whcomm{[We may add a subscript $F$ to $\|\cdot\|$. Otherwise, reader may think it denotes a spectral norm.]}{}In many practical settings, the proximal mapping either has a closed-form solution or can be solved efficiently. Thus,
the algorithm has low per iteration cost and    is applicable for large-scale problems. Furthermore, under the assumptions that $f$ is convex, Lipschitz-continuously differentiable with Lipschitz constant $L$, $g$ is convex, and $F$ is coercive, the proximal gradient method converges  on the order of $O(1/k)$~\cite{BecTeb2009fista,Beck2017}.
Note that the convergence rate of the proximal gradient method is not optimal and algorithms achieving the optimal $O(1/k^2)$ \cite{Darzentas1983Problem,Nesterov83} convergence rate can be developed based on certain acceleration schemes.
In the paper\cite{BecTeb2009fista}, Beck and Teboulle present an accelerated proximal gradient method (known as FISTA) based on the Nesterov momentum technique. The algorithm consists of the following steps
\begin{align*}
\begin{cases}
 \eta_{y_k} = \argmin_{\eta\in\R^{n\times p}}\langle \nabla f(y_k),\eta\rangle+\frac{1}{2 \mu}\|\eta\|_{\mathrm{F}}^2+g(y_k+\eta)\\
 x_{k+1} = y_k+\eta_{y_k}\\
 t_{k+1} = \frac{\sqrt{4t_k^2+1}+1}{2}\\
 y_{k+1} = x_{k+1}+\frac{t_{k}-1}{t_{k+1}}(x_{k+1}-x_k).
 \end{cases}\numberthis\label{eq:APG}
\end{align*}
%For general convex optimization problems, the convergence rate of the proximal gradient method is $O(1/k)$ while the convergence rate of FISTA is $O(1/k^2)$ \cite{BecTeb2009fista}.
Under the same conditions as in the convergence analysis of the proximal gradient method,  FISTA
been proven to converge on the order of $O(1/k^2)$ \cite{BecTeb2009fista}.

In the work of Chen \etal \cite{CMSZ2019}, the Manifold Proximal Gradient method (ManPG)  is proposed to solve \eqref{eq:optimization}. The structure of the algorithm is overall  is similar to \eqref{eq:PG}, except that a subproblem  constrained to the tangent space is solved. More precisely, the following constrained optimization problem is first solved to compute the search direction,
\begin{align*}
\eta_{x_k} = \argmin_{\eta\in\T_{x_k}\M}\langle\grad f(x_k),\eta\rangle+\frac{1}{2\mu}\|\eta\|_{\whsecrev{}{W_{x_k}}}^2+g(x_k+\eta)
,\numberthis\label{eq:RPG}
\end{align*}
{where $\|\eta\|_{\whsecrev{}{W_{x}}}^2=\langle\eta,\whsecrev{}{W_{x}}\eta\rangle$ with $\whsecrev{}{W_{x}}:\T_x\M\rightarrow\T_x\M$ being a symmetric, positive definite linear operator.} Here we describe the proximal subproblem in a more general form by introducing a weight operator. As will be seen in the simulations, a simple diagonal weight that is computed adaptively can help improve the convergence of the algorithms. It is trivial that when $\whsecrev{}{W_{x}}$ is an identity operator,  \eqref{eq:RPG} reduces to the standard proximal subproblem considered in the work of Chen \etal \cite{CMSZ2019}.
After the search direction is found, a new estimate is then computed via backtracking and retraction. Since $g(x)$ is a convex function and $\T_x\M$ is a linear subspace, \eqref{eq:RPG} is indeed a convex programming. Thus there are computationally efficient algorithms for this problem. We will return to this issue later in Section~\ref{sec:SSN}.

The global convergence of the algorithm has been established in the work of Chen \etal \cite{CMSZ2019}. More precisely, the authors show
that the norm of the search direction computed from the Riemannian proximal mapping goes to
zero. In addition, if there exists a point such that the search direction from this point vanishes, then this point must be a critical point. 

\whcomm{[Note that for this sparse PCA problem, the Lipschitz constant $L$ is easy to estimate and the estimation works well in practice. However, it may not be the case for other problems. Add an adaptive scheme for $\mu$? But it is fine to me not to add such scheme since we focus on a particular problem. (Same for $N$.)]}{}

\whcomm{[Change the name from ARPG to AManPG.]}{}
\begin{algorithm}[ht!]
\caption{Accelerated Manifold Proximal Gradient Method (AManPG)}
\label{alg:ARPG}
\begin{algorithmic}[1]
\Require Lipschitz constant $L$ on $\nabla f$, parameter $\mu \in (0, 1 / L]$ in the proximal mapping, line search parameter $\sigma \in (0, 1)$, shrinking parameter in line search $\nu \in (0, 1)$, positive integer $N$ for safeguard;
\State $t_0 = 1$, $y_{0} = x_0$, $z_0 = x_0$;
\For {$k = 0, \ldots$}
%\State Set $\tilde{\xi}_{k - 1} = \frac{t_{k - 1} - 1}{t_{k}} \xi_{k - 1}$ and $y_{k} = R_{x_k}\left( \tilde\xi_{k-1}\right)$
\If {$\mod(k, N) = 0$} \Comment{Invoke safeguard every $N$ iterations}
\State Invoke Algorithm~\ref{alg:Safeguard}: $[z_{k + N}, x_k, y_k, t_k] = Alg\ref{alg:Safeguard}(z_{k}, x_k, y_k, t_k, F(x_k))$;
\EndIf
\State Compute
\begin{equation*}
\eta_{y_k} = \argmin_{\eta \in \T_{y_k} \mathcal{M}} \inner[]{\grad f(y_k)}{\eta} + \frac{1}{2 \mu} \|\eta\|_{\whsecrev{}{W_{y_k}}}^2 + g(y_k + \eta);
\end{equation*}
%\State Set $\alpha = 1$;
%\While {$F(R_{y_k}(\alpha\eta_{y_k})) > F(y_k) - \sigma \alpha \|\eta_{y_k}\|^2$}
%\State $\alpha = \nu \alpha$;
%\EndWhile
\State $x_{k+1} = R_{y_k}(\eta_{y_k})$;
\State $t_{k + 1} = \frac{\sqrt{4 t_k^2 + 1} + 1}{2}$;
\State Compute
$$
y_{k + 1} = R_{x_{k + 1}}\left( \frac{1 - t_k}{t_{k+1}} R_{x_{k+1}}^{-1}(x_k) \right);
$$
\EndFor
\end{algorithmic}
\end{algorithm}

\begin{algorithm}[ht!]
\caption{Safeguard for Algorithm 1}
\label{alg:Safeguard}
\begin{algorithmic}[1]
\Require $(z_{k}, x_k, y_k, t_k, F(x_k))$;
\Ensure $[z_{k + N}, x_k, y_k, t_k]$;
\State \label{alg:Safeguard:st1} Compute
\begin{equation*}
\eta_{z_k} = \argmin_{\eta \in \T_{z_k} \mathcal{M}} \inner[]{\grad f(z_k)}{\eta} + \frac{1}{2 \mu} \|\eta\|_{\whsecrev{}{W_{z_k}}}^2 + g(z_k + \eta);
\end{equation*}
\State Set $\alpha = 1$;
\While {$F(R_{z_k}(\alpha\eta_{z_k})) > F(z_k) - \sigma \alpha {\|\eta_{z_k}\|_{\mathrm{F}}^2}$ \kwcomm{[Double check: Should be F norm here]} } \whcommsec{[Yes, I used F norm.]}{}  \label{alg:Safeguard:st2}
\State $\alpha = \nu \alpha$;
\EndWhile
\If {$F(R_{z_k}(\alpha\eta_{z_k})) < F(x_k)$} \Comment{Safeguard takes effect} \label{alg:Safeguard:st3}
\State $x_k = R_{z_k}(\alpha\eta_{z_k})$, $y_k = R_{z_k}(\alpha\eta_{z_k})$, and $t_k=1$;
\Else
\State $x_k$, $y_k$ and $t_k$ keep unchanged;
\EndIf \label{alg:Safeguard:st4}
\State $z_{k+N} = x_k$; \Comment{Update the compared iterate;}
\end{algorithmic}
\end{algorithm}

Inspired by the works \cite{CMSZ2019,BecTeb2009fista},  the goal of this paper is to extend FISTA to the Riemannian setting for the  optimization problem \eqref{eq:optimization}. The algorithm, dubbed Accelerated Manifold Proximal Gradient method (AManPG), is presented in Algorithm~\ref{alg:ARPG}. 
According to the substitution rule provided in \eqref{eq:substitution}, the second line of \eqref{eq:APG} can be replaced by $R_{y_k}(\eta_{y_k})$, giving the 7th step of Algorithm~\ref{alg:ARPG}.  Moreover, the 9th step in Algorithm~\ref{alg:ARPG} is obtained through the following replacement:
\begin{align*}
y_{k+1} = \underbrace{x_{k+1}+\frac{1-t_k}{t_{k+1}}\underbrace{(x_k-x_{k+1})}_{\mbox{replaced by }R_{x_{k+1}}^{-1}(x_k)}}_{\mbox{replaced by }R_{x_{k+1}}\left(\frac{1-t_k}{t_{k+1}} R_{x_{k+1}}^{-1}(x_k)\right)},
\end{align*}
where the first replacement guarantees that $R_{x_{k+1}}^{-1}(x_k)$ is a tangent vector in $\T_{x_{k+1}}\M$. 

Furthermore, since we are dealing with a non-convex optimization problem, the convergence of the Riemannian version of \eqref{eq:APG} is not guaranteed, even for the convergence to a stationary point as the function value of the iterate does not monotonically decrease. \whcomm{[an accumulation point or a stationary point?]}{} Therefore, a safeguard strategy \kw{via restarting} is introduced in Algorithm~\ref{alg:ARPG} to monitor the progress of the algorithm in every $N$ iterations. Whenever the safeguard rule is violated, the algorithm will  be restarted. \kw{It is worth noting that the idea of restarting has also been used in the Euclidean setting to suppresses the oscillatory behaviour of the accelerated proximal gradient methods, see for example the work of O'Donoghue\cite{DC2015AdapFISTA}.}

When we apply Algorithm~\ref{alg:ARPG} to the sparse PCA problem \eqref{eq:sparsePCA}, the computation of the retraction is already given in \eqref{eq:computeRetraction}.  
To compute the inverse of the retraction we first note that  $R_X^{-1}(Y)$ exists when $Y$ is not far from $X$ owing to the local diffeomorphism property of retraction. Letting $\eta_X=R_X^{-1}(Y)$, by \eqref{eq:polar}, we have $\eta_X = YS-X$ for $S=(I_p+\eta_X^T\eta_X)^{1/2}$. Combining the fact $\eta_X\in\T_X\St(p,n)$ and \eqref{eq:tangent} yields
\begin{align*}
(X^TY)S+S(Y^TX) = 2I_p.\numberthis\label{eq:lyap}
\end{align*}
{This is a Lyapunov equation which can be computed by the Bartels-Stewart algorithm using $O(p^3)$ flops \cite{BaSt1972a}. } Once $S$ is computed from \eqref{eq:lyap}, inserting it back into $\eta_X = YS-X$ gives $R_X^{-1}(Y)$.
It is worth noting that the additional computational cost incurred by the Lyapunov equation is marginal since it is very typical that $p\ll n$ in the sparse PCA problem.
%%%%%%
\subsection{Computing the diagonal weight}\label{sec:weight}
In this paper we will restrict our attention to the diagonal weight for two reasons. Firstly, it is easy to compute for the sparse PCA problem. Secondly, the proximal subproblem \eqref{eq:RPG} with a diagonal weight can be solved as efficiently as that without a weight.

Roughly speaking, we will extract a diagonal weight from the expression of the Riemannian Hessian of $f$ in each iteration. In particular, when applying the Riemannian proximal gradient methods (including ManPG and AManPG) to the sparse PCA problem \eqref{eq:sparsePCA}, a diagonal weight can be computed in the following way. Noting $f(X)=-\|AX\|_F^2$ in \eqref{eq:sparsePCA}, by \eqref{eq:grad} and \eqref{eq:proj}, we have 
\begin{align*}
\grad f(X) &= \P_{\T_X\M}(-2A^TAX)\\
&=-2A^TAX +2X(X^TA^TAX).
\end{align*}
It follows that 
\begin{align*}
\D\grad f(X)[\eta_X]& = -2A^TA\eta_X +2\eta_X(X^TA^TAX)\\
&+2X(\eta_X^TA^TAX+X^TA^TA\eta_X),\quad\forall \eta_X\in\T_X\St(p,n).
\end{align*}
Noting that 
$
\P_{\T_X\St(p,n)}(X(\eta^TA^TAX+X^TA^TA\eta)) =0,
$
it follows from \eqref{eq:hessian} that
\begin{align*}
\Hess f(X)[\eta_X]=\P_{\T_X\St(p,n)}(-2A^TA\eta_X +2\eta_X(X^TA^TAX)).
\end{align*}
In the Riemannian Newton's method, the weight operator should be chosen in a way such that
\begin{align*}
\langle \eta_X,W\eta_X\rangle=\langle\eta_X,\Hess f(X)[\eta_X]\rangle =\langle\eta_X,-2A^TA\eta_X +2\eta_X(X^TA^TAX)\rangle,
\end{align*}
where the second equality follows from the fact $\eta_X\in \T_X\St(p,n)$. After vectorization we 
can rewrite the third inner product as 
\begin{align*}
\langle\eta_X,-2A^TA\eta_X +2\eta_X(X^TA^TAX)\rangle =\langle \mathrm{vec}(\eta_X),J\mathrm{vec}(\eta_X)\rangle,
\end{align*}
where $J$ is an $np\times np$ matrix given by 
\begin{align*}
J = -2I_p\otimes (A^TA)+2(X^TA^TAX)\otimes I_n.
\end{align*}
\whcommsec{[Add discussion about an efficient way to compute the diagonal entries.]}{} \kwcomm{[To avoid something like {\tt repmat}, only the formula in the matrix form is given from which it should be easy to see the efficient implementation]}{} \whcommsec{[Fine]}{} Since a diagonal weight is sought here, a natural choice is to set $W$ to be the diagonal part of $J$, given by 
\begin{align*}
\diag(J) = -2(D_1-D_2),
\end{align*}
where
\begin{align*}
D_1=\begin{bmatrix}
\diag(A^TA)\\
&\diag(A^TA)\\
& & \ddots\\
&&&\diag(A^TA)
\end{bmatrix}
\end{align*} 
and
\begin{align*}
D_2=
\begin{bmatrix}
(X^TA^TAX)_{11}I_n\\
&(X^TA^TAX)_{22}I_n\\
&&\ddots\\
&&&(X^TA^TAX)_{pp}I_n
\end{bmatrix}.
\end{align*}
Furthermore, in order to make sure  $W$ is positive definite, we use the following  modification in \eqref{eq:RPG},
\begin{align*}
W =  \max\{\diag(J), \tau I_{np}\},\numberthis\label{eq:weight}
\end{align*}
where $\tau>0$ is a tuning parameter. \whcommsec{[What I am using now is $W = \max(\diag(J), \tau I_{np})$.]}{} \kwcomm{[Done.]}{}
%%%%%
\subsection{Outline of the semi-smooth Newton method for \eqref{eq:RPG}} \label{sec:SSN}
As suggested in the work\cite{CMSZ2019}, the proximal subproblem can be solved efficiently by the semi-smooth Newton method. To keep the presentation self-contained, this section outlines the key ingredients for  applying the semi-smooth Newton method to solve \eqref{eq:RPG}. Interested readers can find more details about the semi-smooth Newton method in the work \cite{CMSZ2019,XiLiWeZh2018a,LiSunToh18} and references therein. Overall, semi-smooth Newton method is about solving a system of nonlinear equations based on the notion of the generalized Jacobian. Thus to apply the semi-smooth Newton method, we need to reformulate an optimization problem as a system of nonlinear equations. This   can usually be achieved by considering the KKT conditions or the fixed point mappings.

Considering the sparse PCA problem \eqref{eq:sparsePCA}, we can first rewrite the Riemannian proximal subproblem \eqref{eq:RPG} as 
\begin{align*}
\eta^* = \argmin_\eta\langle \grad f(X),\eta\rangle+\frac{1}{2\mu}\langle \eta,W\eta\rangle+g(X+\eta)\quad\mbox{subject to}\quad \eta\in\T_X\St(p,n),\numberthis\label{eq:subStiefel}
\end{align*}
where we omit the subscripts for conciseness.  As in the work \cite{CMSZ2019}, let $\mathcal{A}:\R^{n\times p}\rightarrow\R^{p\times p}$ be a linear operator defined by $\mathcal{A}(\eta)=X^T\eta+\eta^TX$. Noting the expression of $\T_X\St(p,n)$ in \eqref{eq:tangent}, it is not hard to  see that the KKT condition for \eqref{eq:subStiefel} is given by 
\begin{align*}
\begin{cases}
\partial_\eta \mathcal{L}(\eta,\lambda)=0\\
\mathcal{A}(\eta)=0,
\end{cases}\numberthis\label{eq:KKT}
\end{align*}
where $\mathcal{L}(\eta,\lambda)$ the \whcommsec{[This is a Lagrangian function, not augmented Lagrangian, right?]augmented Lagrangian}{Lagrangian function} associated with \eqref{eq:subStiefel},
\begin{align*}
\mathcal{L}(\eta,\lambda)=\langle \grad f(X),\eta\rangle+\frac{1}{2\mu}\langle \eta,W\eta\rangle+g(X+\eta)-\langle\lambda,\mathcal{A}(\eta)\rangle.\numberthis\label{eq:lagrangian}
\end{align*}
From the first equation of \eqref{eq:KKT}, we have
\begin{align*}
\eta= \Prox_{ug}^{W}\left(X-\mu W^{-1}(\grad f(X)-\mathcal{A}^*\lambda)\right)-X,
\numberthis\label{eq:lam2eta}
\end{align*}
where 
\begin{align*}
 \Prox_{ug}^{W}(Z) = \argmin_{V\in\R^{n\times p}} \frac{1}{2}\|V-Z\|_{W}^2+\mu g(V)\numberthis\label{eq:scaledProx}
\end{align*}
denotes the scaled proximal mapping \cite{LeeSunSau14}, and $\mathcal{A}^*$ denotes the adjoint of $\mathcal{A}$. Substituting \eqref{eq:lam2eta} into the second equation of \eqref{eq:KKT} yields that 
\begin{align*}
\Psi(\lambda):=\mathcal{A}\left( \Prox_{ug}^{W}\left(X-\mu W^{-1}(\grad f(X)-\mathcal{A}^*\lambda)\right)-X\right)=0,\numberthis\label{eq:noneq}
\end{align*}
which is a system of nonlinear equations with respect to $\lambda$. Thus, to compute the solution to the proximal subproblem \eqref{eq:subStiefel}, we can first find the root of the nonlinear system \eqref{eq:noneq} and then substitute it back to \eqref{eq:lam2eta} to obtain $\eta^*$.

When $W$ is a diagonal weight operator, the nonlinear system \eqref{eq:noneq} can be solved efficiently by the semi-smooth Newton method. Let $\lambda_k$ be the current estimate of the solution to \eqref{eq:noneq}. As in the Newton method, the key step in the semi-smooth Newton method is to compute a search direction by solving the following linear system 
\begin{align*}
J_\Psi(\lambda_k)[d] = -\Psi(\lambda_k),
\end{align*}
where $J_\Psi(\lambda_k)$ is generalized Jacobian of $\Psi$.  Note that when $W$ is a diagonal operator and $g(V)=\|V\|_1$, it is well-known that the solution to the scaled proximal mapping \eqref{eq:scaledProx} can be computed by thresholding each entry of $Z$. Moreover, by the chain rule, we have
\begin{align*}{
J_\Psi(\lambda_k)[d] = \mathcal{A}\left(\partial\Prox_{\mu g}^{W}\left(X-\mu W^{-1}(\grad f(X)-\mathcal{A}^*\lambda_k)\right)\circ\left(\mu W^{-1}\mathcal{A}^*d\right)\right),}
\end{align*}
where $\partial\Prox_{\mu g}^{W}(\cdot)$ denotes the \whfirrev{}{generalized} Clarke subdifferential of $\Prox_{\mu g}^{W}(\cdot)$ and $\circ$ denotes the entrywise product of two matrices. Once again,  when $W$ is a diagonal operator and $g(V)=\|V\|_1$ the \whfirrev{}{generalized} Clarke subdifferential of $\Prox_{\mu g}^{W}(\cdot)$ can also be computed in an entrywise manner \cite{XiLiWeZh2018a,LiSunToh18,Clarke1990a}. 
Note that in our implementations of the semi-smooth Newton method, we follow the algorithmic framework in the work of Xiao \etal\cite{XiLiWeZh2018a}, where a safeguard step is also introduced.

\whfirrev{}{
The computational complexity is measured by flop counts. A flop is a floating point operation \cite[Section 1.2.4]{GV96}. The dominant computational costs in one evaluation of $\Psi(\lambda)$ and $J_\Psi(\lambda_k)[d]$ are respectively $4 n p^2$ and $6 np^2$ flops. Therefore, the total computational costs in the semi-smooth Newton method is on the order of $n p^2$ where  the coefficient depends on the number of iterations (usually 2 or 3 iterations). Note that one evaluation of $f$, $\nabla f$, $\P_{\T_x \St(p, n)}$, $R$ and $R^{-1}$ takes $2 m n p$, $2 m n p$, $4 n p^2$, $4 n p^2 + O(p^3)$ and $4 n p^2 + O(p^3)$, respectively. Therefore, the overall complexity of Algorithm~\ref{alg:ARPG} is on the order of $k m n p + k n p^2 + k p^3$, where the value of $k$ depends on the number of outer/inner iterations. %If $m \gg n$, then one can first compute $A^T A$, which yield that the complexities for one evaluation of $f$, $\nabla f$ are both $O(n^2 p)$. It follows that the complexity of Algorithm~\ref{alg:ARPG} is $2 m n^2 + O(k n^2 p + k n p^2 + k p^3)$.
}

%%%%%%
\subsection{Convergence analysis \kwcomm{[need minor adjustment after analysis is double checked]}{} \whcommsec{[Done]}{}}
In this section we show that any accumulation point of the sequence $\{z_k\}$ generated by Algorithm~\ref{alg:ARPG} is a stationary point. In other words, if $z_*$ is an accumulation point of $\{z_k\}$, then there holds $0 \in \P_{\T_{z_*} \mathcal{M}} \partial F(z_*)$, where $\partial F(x)$ denotes the \whfirrev{}{generalized Clarke} subgradient of $F$ at $x$ and $\P_{\T_{z_*} \mathcal{M}}$ denotes the orthogonal projection to the tangent space of $\mathcal{M}$ at $z$. 
%\whcommsec{}{The function value decrement needed for the proof can be established similarly to that in \cite{CMSZ2019} by assuming that the smallest eigenvalue of the weight matrix $W$ at all $z_k$ are bounded away from zero.} 
\whfirrev{}{In the work of Chen \etal~\cite{CMSZ2019}, it has been shown that the search direction computed in ManPG converges to zero, and if the search direction is zero at $x$, then $x$ is a stationary point. To the best of our knowledge, this does not directly imply that any accumulation point of the iterates generated by the algorithms is a stationary point. For the Euclidean case, such a result can be found in the work~\cite{RocWet1998,LL2015}. For the Riemannian case, 
}
\whfirrev{}{
we complete the stationary point analysis by showing that if $(z_k, u_k)$ is a sequence such that $u_k\in \P_{\T_{z_k} \mathcal{M}} \partial F(z_k+\eta_{z_k})$, $$\mbox{$z_k\rightarrow z_*$,~ $F(z_k)\rightarrow F(z_*)$,~ $\eta_{z_k}\rightarrow 0$,~ and $u_k\rightarrow 0$,}$$ then  we have $ 0 \in \P_{\T_{z_*} \mathcal{M}} \partial F(z_*)$.
}

%Here we prove the other piece of result to complete the stationary point analysis.  The result can be roughly expressed as: letting $(z_k, u_k)$ be a sequence such that $u_k\in \P_{\T_{z_k} \mathcal{M}} \partial F(z_k+\eta_{z_k})$, if $$\mbox{$z_k\rightarrow z_*$,~ $F(z_k)\rightarrow F(z_*)$,~ $\eta_{z_k}\rightarrow 0$,~ and $u_k\rightarrow 0$,}$$ then  we have $ 0 \in \P_{\T_{z_*} \mathcal{M}} \partial F(z_*).$ 
The analysis relies on the following assumptions.
\begin{assumption} \label{as3}
The function $F$ is coercive, i.e., $F(x) \rightarrow +\infty$ as $\|x\|_{\mathrm{F}} \rightarrow \infty$.
\end{assumption}

\begin{assumption} \label{as1}
The function $f:\mathbb{R}^{n \times p} \rightarrow \mathbb{R}$ is Lipschitz continuously differentiable.
\end{assumption}

\begin{assumption} \label{as2}
The function $g:\mathbb{R}^{n \times p} \rightarrow \mathbb{R}$ is continuous and convex.
\end{assumption}

\begin{assumption} \label{as4}
There exists two positive constants $0 < \kappa \leq \tilde{\kappa}$ such that the weight matrix $W$ at $z_k$, denoted by $W_{z_k}$, satisfies that the eigenvalues of $W_{z_k}$ are between $\kappa$ and $\tilde{\kappa}$ for all $k$.
\end{assumption}

It is worth mentioning that Assumption~\ref{as4} is not a stringent assumption. For example, the diagonal weight constructed for the sparse PCA problem in \eqref{eq:weight} satisfies this assumption since the Stiefel manifold is compact and $J$ is continuous over this manifold.

\begin{lemma} \label{le1}
Suppose Assumptions~\ref{as3}, \ref{as1} and \ref{as2} hold. Then
\begin{enumerate}
\item \label{le1:1} the sublevel set $\Omega_{x_0} = \{x \in \mathcal{M} \mid F(x) \leq F(x_0) \}$ is bounded;
\item \label{le1:2} $F$ is Lipschitz continuous in $\Omega_{x_0}$ and bounded from below;
\item \label{le1:3} there exists a constant $M$ such that $\max_{x \in \Omega_{x_0}} \max_{v \in \partial F(x)} \|v\|_{\mathrm{F}} \leq M$.
\end{enumerate}
\end{lemma}
\begin{proof}
It follows from Assumption~\ref{as3} that $\Omega_{x_0}$ is bounded. The convexity of $g$ implies that $g$ is locally Lipschitz continuous \cite[Theorem~4.1.1]{BoLe2006a}. % Theorem 4.1.1 in Convex Analysis and Nonlinear Optimization second edition
Therefore, $g$ is Lipschitz continuous in the compact set $\Omega_{x_0}$. Combining this result with Assumption~\ref{as1} yields that $F$ is Lipschitz continuous in $\Omega_{x_0}$.
Since $\Omega_{x_0}$ is compact, there exists a ball with radius $R$, $B(0, R)$, such that $\Omega_{x_0} \subset B(0, R)$. We have
\begin{align*}
|F(x) - F(x_0)| \overset{\text{Lipschitz continuity of $F$}}{\leq}& L_F \|x - x_0\|_{\mathrm{F}} \leq L R,
\end{align*}
which yields $F(x) \geq F(x_0) - L_F R$ for all $x \in \Omega_{x_0}$. For any $x \notin \Omega_{x_0}$, we have $F(x) > F(x_0)$. Therefore, $F(x)$ is bounded from below.
By the work \cite[Proposition 2.1.2]{Clarke1990a}, %Proposition 2.1.2 in Optimization and nonsmooth analysis by Clarke
the Lipschitz constant $L_F$ of $F$ in $\Omega_{x_0}$ satisfies that $\max_{x \in \Omega_{x_0}} \max_{v \in \partial F(x)} \|v\|_{\mathrm{F}} \leq L_F$.
\end{proof}

Since the subscripts of the sequence $\{z_k\}$ in Algorithm~\ref{alg:ARPG} are multiple of $N$, we use $\{\tilde{z}_i\}$ to denote $\{z_k\}$, where $\tilde{z}_i = z_{i N}$.
If $W_{z_k} \equiv I$, then the subproblem in Step~\ref{alg:Safeguard:st1} of Algorithm~\ref{alg:Safeguard} is the same as that in the work\cite{CMSZ2019}, and therefore related results from the work \cite{CMSZ2019}, stated in Lemma~\ref{le2}, hold.
Under Assumption~\ref{as4}, we claim that Lemma~\ref{le2}  can still be applied here without assuming $W_{z_k} \equiv I$. The proof is given in Appendix~\ref{app1}.\footnote{Note that the proof of Lemma~\ref{le2} in the work \cite{CMSZ2019}  essentially relies on \ref{le1:1} and \ref{le1:2} of Lemma~\ref{le1}.}

%Therefore, related  results from  \cite{CMSZ2019} can still be applied here.
%In particular, we will use the following lemma.
\begin{lemma} \label{le2}
The following properties hold:
\begin{enumerate}
\item \label{le2:1} There exist constants $\bar{\alpha} > 0$ and $\bar{\beta} > 0$ such that for any $0 < \alpha \leq \min(1, \bar{\alpha})$, the sequence $\{\tilde{z}_i\}$ satisfies:
\begin{equation*}
F(R_{\tilde{z}_i}(\alpha \eta_{\tilde{z}_i})) - F(\tilde{z}_i) \leq - \bar{\beta} \whfirrev{}{\alpha} \|\eta_{\tilde{z}_i}\|_{\mathrm{F}}^2.
\end{equation*}
\item \label{le2:2} If $\eta_{\tilde{z}_i} = 0$, then $\tilde{z}_i$ is a stationary point of Problem~\eqref{eq:optimization}.
%\item \label{le2:3} The sequence $\{\eta_{z_k}\}$ satisfies $\lim_{k\rightarrow \infty} \|\eta_{z_k}\| = 0$.
\end{enumerate}
\end{lemma}
\whfirrev{}{The two items of Lemma~\ref{le2} follow from the work~\cite[Lemmas~5.2 and~5.3]{CMSZ2019}.}
The first item of Lemma~\ref{le2} implies that the line search  in Step~\ref{alg:Safeguard:st2} of Algorithm~\ref{alg:Safeguard} terminates in finite iterations. Therefore, Algorithm~\ref{alg:ARPG} is well-defined.

%Lemma~\ref{le3} shows that Algorithm~\ref{alg:ARPG} is also a decreasing algorithm with respect to the sequence $\{\tilde{z}_i\}$.
\begin{lemma} \label{le3}
Suppose Assumptions~\ref{as3}, \ref{as1}, ~\ref{as2} and~\ref{as4} hold. Then
\begin{enumerate}
\item \label{le3:1} $F(\tilde{z}_{i+1}) < F(\tilde{z}_i)$. Therefore, $\{\tilde{z}_i\} \subset \Omega_{x_0}$.
\item \label{le3:2} The sequence $\{\eta_{\tilde{z}_i}\}$ satisfies $\lim_{i\rightarrow \infty} \|\eta_{\tilde{z}_i}\|_{\mathrm{F}} = 0$.
\end{enumerate}
\end{lemma}
\begin{proof}
By Steps~\ref{alg:Safeguard:st3} to~\ref{alg:Safeguard:st4} of Algorithm~\ref{alg:Safeguard}, we have $F(\tilde{z}_{i+1}) \leq F(R_{\tilde{z}_i}(\alpha_i \eta_{\tilde{z}_i}))$, \whfirrev{}{where $\alpha_i$ denotes the accepted step size}. Combining it with \ref{le2:1} of Lemma~\ref{le2} yields $F(\tilde{z}_{i+1}) < F(\tilde{z}_i)$. Since $F$ is bounded from below by \ref{le1:2} of Lemma~\ref{le1} and $\{F(\tilde{z}_i)\}$ is decreasing, we have $\lim_{i \rightarrow \infty} F(\tilde{z}_i) - F(R_{\tilde{z}_i}(\alpha \eta_{\tilde{z}_i})) = 0$. Combining it with \ref{le2:1} of Lemma~\ref{le2} yields $\lim_{k\rightarrow \infty} \alpha_i \|\eta_{\tilde{z}_i}\|_{\mathrm{F}}^2 = 0$. \whfirrev{}{
By~\ref{le2:1} of Lemma~\ref{le2} and the backtracking in Step~\ref{alg:Safeguard:st2} of Algorithm~\ref{alg:Safeguard}, we have that $\alpha_i \geq \min(1, \nu \bar{\beta} \bar{\alpha} / \sigma)$ for all $i$. Therefore, $\lim_{k\rightarrow \infty} \|\eta_{\tilde{z}_i}\|_{\mathrm{F}} = 0$.
}
\end{proof}

The norms of $\eta_{\tilde{z}_i}$ go to zero by \ref{le3:2} of Lemma~\ref{le3}. \whfirrev{}{The following theorem further establishes that 0 is in the subgradient of any accumulation point of $\tilde{z}_i$.}
\begin{theorem} \label{th1}
Suppose Assumptions~\ref{as3}, \ref{as1}, \ref{as2} and~\ref{as4} hold.
Let $z_*$ be any accumulation point of the sequence $\{\tilde{z}_{i}\}$. We have
\begin{equation*}
0 \in \P_{\T_{z_*} \mathcal{M}} \partial F(z_*).
\end{equation*}
\end{theorem}
%Note that $\{\tilde{z}_i\}$ is a subsequence of $\{x_k\}$. This theorem also implies that any accumulation point of a subsequence of $\{x_k\}$ is a stationary point.
\begin{proof}
By Step~\ref{alg:Safeguard:st1} of Algorithm~\ref{alg:Safeguard}, we have
\begin{align*}
\eta_{\tilde{z}_i} = \argmin_{\eta \in \T_{\tilde{z}_i} \mathcal{M}} \inner[]{\grad f(\tilde{z}_i)}{\eta} + \frac{1}{2 \mu} \|\eta\|_{W_{\tilde{z}_i}}^2 + g(\tilde{z}_i + \eta).
\end{align*}
Therefore,
$
0 \in \grad f(\tilde{z}_i) + \frac{1}{\mu} W_{\tilde{z}_i} \eta_{\tilde{z}_i} + \P_{\T_{\tilde{z}_i} \mathcal{M}} \partial g(\tilde{z}_i + \eta_{\tilde{z}_i})
$
which yields
\begin{equation*}
- \grad f(\tilde{z}_i) + \grad f(\tilde{z}_i + \eta_{\tilde{z}_i}) - \frac{1}{\mu} W_{\tilde{z}_i} \eta_{\tilde{z}_i} \in \P_{\T_{\tilde{z}_i} \mathcal{M}} \partial F(\tilde{z}_i + \eta_{\tilde{z}_i}).
\end{equation*}
%\whcomm{[In preliminaries on manifold section, we could introduce the Riemannian gradient $\grad f(x)$.]}{}
Thus, there exists a sequence $\xi_{i} \in \N_{\tilde{z}_i} \mathcal{M}$ such that
\begin{equation*}
- \grad f(\tilde{z}_i) + \grad f(\tilde{z}_i + \eta_{\tilde{z}_i}) - \frac{1}{\mu} W_{\tilde{z}_i} \eta_{\tilde{z}_i} + \xi_{i} \in \partial F(\tilde{z}_i + \eta_{\tilde{z}_i}),
\end{equation*}
where $\N_{\tilde{z}_i} \mathcal{M}$ denotes the normal space of $\mathcal{M}$ at $\tilde{z}_i$.
Let $\tilde{z}_{i_j}$ be the subsequence converging to $z_*$. We have
\begin{equation*}
- \grad f(\tilde{z}_{i_j}) + \grad f(\tilde{z}_{i_j} + \eta_{\tilde{z}_{i_j}}) - \frac{1}{\mu} W_{\tilde{z}_i} \eta_{\tilde{z}_{i_j}} + \xi_{i_j}  \in \partial F(\tilde{z}_{i_j} + \eta_{\tilde{z}_{i_j}}).
\end{equation*}
By~\ref{le1:3} of Lemma~\ref{le1}, we have that $\|\xi_{i_j}\|_{\mathrm{F}} < M$ for all $j$. Therefore, there exists a converging subsequence $\{\xi_{i_{j_s}}\}$ and let $\xi_*$ denote its limit point.
It follows from \ref{le3:2} of Lemma~\ref{le3} and Assumptions~\ref{as1} and~\ref{as4} that
\begin{equation*}
- \grad f(\tilde{z}_{i_{j_s}}) + \grad f(\tilde{z}_{i_{j_s}} + \eta_{\tilde{z}_{i_{j_s}}}) - \frac{1}{\mu} W_{\tilde{z}_i} \eta_{\tilde{z}_{i_{j_s}}} + \xi_{i_{j_s}} \rightarrow \xi_* \hbox{ and } \tilde{z}_{i_{j_s}} + \eta_{\tilde{z}_{i_{j_s}}} \rightarrow z_*,
\end{equation*}
as $s \rightarrow \infty$. Then by the work \cite[Remark 1(ii)]{BoSaTe2014a}, it holds that
\begin{equation} \label{eq:e1}
\xi_* \in \partial F(z_*).
\end{equation}
\whfirrev{}{
Note that in~\eqref{eq:e1}, $F$ is viewed as a function on a Euclidean space and $\partial$ denotes the (non-Riemannian) generalized Clarke subdifferential.
}
Since the projection $\P_{N_x \mathcal{M}}$ is smooth with respect to the root $x$, we have that
$$
\xi_{i_{j_s}} = \P_{\N_{\tilde{z}_{i_{j_s}}} \mathcal{M}} \xi_{i_{j_s}} \rightarrow \P_{\N_{z_*}\mathcal{M}} \xi_* \hbox{ and } \xi_{i_{j_s}} \rightarrow \xi_*,
$$
as $s \rightarrow \infty$.
Therefore, $\P_{\N_{z_*}\mathcal{M}} \xi_* = \xi_*$, which implies $\xi_*$ is in the normal space at $z_*$. It follows from~\eqref{eq:e1} that
$$
0 \in \P_{\T_{z_*}\mathcal{M}} \partial F(z_*),
$$
which completes the proof.
\end{proof}
%%%%%%
%\subsection{Computation details for sparse PCA}

%%%%%%
%\paragraph{Computing the solution of the subproblem} As suggested in CITE, the subproblem in the 6th step of Algorithm~\ref{alg:ARPG} (and the 1st step of Algorithm~\ref{alg:Safeguard}) can be computed efficiently by the semi-smooth Newton method (SSN) developed in CITE. For completeness, we outline the key ideas behind the SSN method for the subproblem corresponding to sparse PCA.
%%%%%%

\section{Numerical Experiments}\label{sec:num}
%%%%%%

\whcomm{[
More comparisons may be added. The following are SPCA models, whose codes are readily available online.

model by Zou, Hastie, Tibshirani. Codes:
\url{http://www2.imm.dtu.dk/projects/spasm/}

model by Hein and T. Buhler. Codes:
\url{https://www.ml.uni-saarland.de/code/sparsePCA/sparsePCA.htm}

8 models: codes:
\url{https://github.com/amjams/spca_am}

]}{}

\whcomm{[A description of ManPG-Ada can be added here.]}{}

This section evaluates the empirical performance of \kw{ AManPG  with and without the diagonal weight 
using the sparse PCA problem \eqref{eq:sparsePCA}, and compare them with the existing methods. 
}
%ManPG-Ada is a variant of ManPG which is also introduced in \cite{CMSZ2019}. It has been observed in \cite{CMSZ2019} that AManPG-Ada can achieve faster convergence than ManPG by adaptively adjusting the constant $\mu$ in~\eqref{eq:RPG}. 
%\kw{Note that  the associated algorithms   using the diagonal weight computed in the way presented in Section~\ref{sec:weight} are denoted by AManPG-D and ManPG-Ada-D, respectively, while AManPG and ManPG-Ada means the algorithms without the diagonal weight (i.e.,  $W=I$ in the Riemannian proximal subproblem).}

\subsection{Testing environment and parameter settings}\label{sec:numerics01}
\kwcomm{}{All the tested algorithms are implemented in the ROPTLIB package \cite{HAGH2016} using C++, with a MATLAB interface.}  The experiments are performed in Matlab R2019a on a 64 bit MacOS Mojave platform with 2.7 Ghz CPU (Intel Core i7), and the source codes for reproducible research can be downloaded at~
\begin{align*}\mbox{\url{https://www.math.fsu.edu/~whuang2/papers/EFROSP.htm}.}
\end{align*}

In this section three different types of data matrices are tested, and they are generated through the following way:
\begin{enumerate}
\item {\bf Random data.} The entries in the data matrix $A$ are drawn from the standard normal distribution~$\mathcal{N}(0, 1)$.
\item {\bf DNA methylation data.} The data is available on the NCBI website with the reference number GSE32393~\cite{Joanna2012}.
\item {\bf Synthetic data.}  As is done in the work of Sj\"ostrand \etal\cite{SCLEE2018},  we first repeat the five principal components (shown in Figure~\ref{figPCs})  $m/5$ times to obtain an $m$-by-$n$ noise-free matrix. Then the data matrix $A$ is created by further adding a random noise matrix, where each entry of the noise matrix is drawn from~$\mathcal{N}(0, 0.25)$. %This data is generated by inspiring the synthetic data in SpaSM library~\cite{SCLEE2018}.
\end{enumerate}
\begin{figure}[ht!]
\centering
\includegraphics[width=0.8\textwidth]{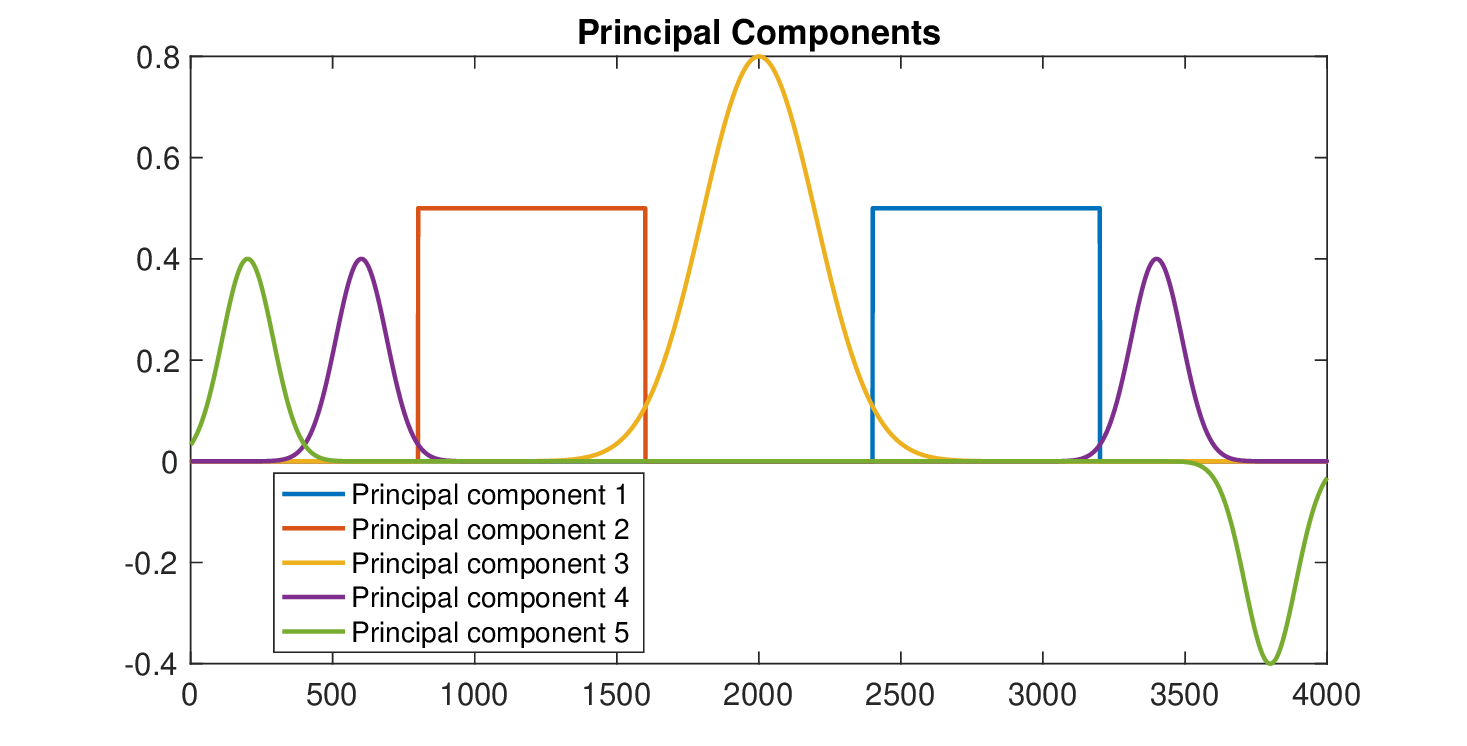}
\caption{
The five principal components used in the synthetic data.
}
\label{figPCs}
\end{figure}
In addition, the  matrices corresponding to the random data and the DNA methylation data are shifted and normalized such that their columns have mean zero and standard deviation one. The matrix for the synthetic data is only normalized such that it columns have standard deviation one since the sparsity over the five principal components needs to be preserved. 

%Unless otherwise indicated, the parameters in ManPG and its variant are set to their default values. 
The parameters  
$\sigma$, $\nu$, $\mu$, and $N$ in AManPG are set to be $10^{-4}$, $0.5$, $1 / (2 \|A\|_2^2)$, and $5$ respectively. When the diagonal weight is used, the parameters $\mu$ and $\tau$ are set to be 1 and 0.1, respectively.
All the tested algorithms terminate when $\|\eta_{z_k}\|_P^2 < \mu n \whfirrev{}{p} 10^{-10}$ or the number of iterations exceeds 10000, where $\|\eta_{z_k}\|_P$ denotes the $\mathrm{F}$-norm for the methods without the diagonal weight and the $W$-norm for the methods with the diagonal weight. The initial guess is constructed from the leading \whfirrev{}{$p$} right singular vectors of the given matrix~$A$. \whfirrev{}{Note that the reported computational time of all the algorithms do not include the computational time for the initial iterate.}

%{\color{red} new section}
\subsection{\kw{Acceleration behavior of AManPG} and influence of the safeguard}\label{sec:numerics02}
\kw{Here  we empirically show that as in the Euclidean case AManPG (with $W=I$ in the Riemannian proximal subproblem) also  achieves  faster convergence than ManPG, and moreover the safeguard in AManPG is able to  stabilize the algorithm while not sacrificing the faster convergence rate. The parameters in ManPG are set to the default values.   Figure~\ref{RPG:TestRestart_f}   contains the comparisons in the three different scenarios. Note that AManPG without the safeguard is  abbreviated as AManPG w/o SG in the figure, while AManPG simply denotes the method with the safeguard here and later. }
%Figure~\ref{RPG:TestRestart_f} shows three plots of function values versus iterations for the ManPG, AManPG, and AManPG without the safeguard (AManPG WO SG). 
When both the AManPG methods with  and without the safeguard converge, they perform similarly as shown in the left plot. This implies that using safeguard in AManPG does not destroy the efficient performance. In addition, AManPG w/o SG may not converge as shown in the middle and right plots. Therefore,  AManPG with the safeguard is preferred since it preserves the global convergence property as ManPG and on  the other hand converges faster than ManPG. 

%{\color{red} review the use of safeguard in Euclidean setting.}

%As shown in the left plot, AManPG with or without using safeguard performs similarly, which implies that using the safeguard does not destroy the efficiency of the 

\begin{figure}[ht!]
\hspace{-0.05\textwidth}\includegraphics[width=1.1\textwidth]{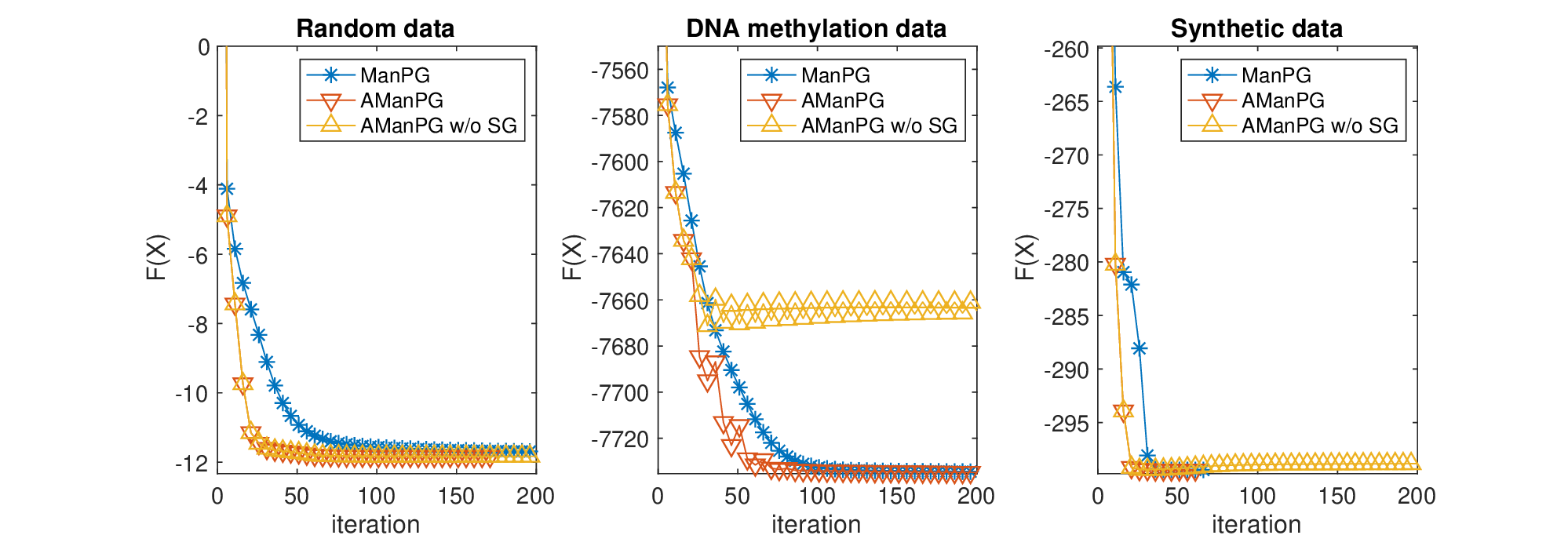}
\caption{
Plots of {\em function values versus iterations} for three typical instances with $\mu = 1 / (1.2 \|A\|_2^2)$. Left: Random data, \whfirrev{}{$n = 3000$, $m = 40$,} $\whfirrev{}{p}=4$, $\lambda=2.5$; Middle: DNA methylation data, \whfirrev{}{$n = 24589$, $m = 113$,} $\whfirrev{}{p}=4$, $\lambda=6$; Right: Synthetic data, \whfirrev{}{$n = 4000$, $m = 400$,} $\whfirrev{}{p} = 5$, $\lambda = 1.5$; The number of restarts in the safeguard in the three tests are 1, 12, and, 3, respectively, from left to right. The values $\|x_{k+1}-x_k\|_F$ of AManPG w/o SG in the middle and right plots stay above 0.22 and 0.07 respectively, up to 10000 iterations.
}
\label{RPG:TestRestart_f}
\end{figure}

\subsection{\kw{Comparisons with other algorithms}}\label{sec:numerics03}
\kw{In this section we compare the performance of AManPG and ManPG-Ada with and without diagonal weight. ManPG-Ada is a variant of ManPG which is also introduced in the work of Chen \etal \cite{CMSZ2019}. It has been observed in the work \cite{CMSZ2019} that AManPG-Ada can achieve faster convergence than ManPG by adaptively adjusting the constant $\mu$ in~\eqref{eq:RPG}.  
The parameters in ManPG-Ada are set to the default values. 
Note that the associated algorithms   using the diagonal weight computed in the way presented in Section~\ref{sec:weight} are denoted by AManPG-D and ManPG-Ada-D, respectively, while AManPG and ManPG-Ada denote the algorithms without the diagonal weight (i.e.,  $W=I$ in the Riemannian proximal subproblem). }
\whcommtrd{}{
These methods are also compared to SOC (splitting method for orthogonality), as Euclidean space based method introduced in the work of Lai \etal~\cite{LO2014}. Since the optimization problem~\eqref{eq:sparsePCA} can be written as
\begin{align} \label{ReWrProb}
\begin{array}{c}
\min_{X, Q, P \in \mathbb{R}^{n \times p}} -\trace(P^T A^T A P) + \lambda \|Q\|_1, \\
\hbox{ s.t. $Q=P, X=P, X^TX = I_p$,}
\end{array}
\end{align}
the SOC method solves~\eqref{ReWrProb} by a three-block ADMM:
\begin{align}
P_{k+1} =& \argmin_{P} -\trace(P^T A^T A P) + \frac{\beta}{2} \|P - Q_k + \Lambda_k\|_{\mathrm{F}}^2 + \frac{\beta}{2} \|P - X_k + \Gamma_k\|_{\mathrm{F}}^2, \label{SOCst1} \\
Q_{k+1} =& \argmin_{Q} \lambda \|Q_1\| + \frac{\beta}{2} \|P_{k+1} - Q + \Lambda_k\|_{\mathrm{F}}^2, \nonumber \\
X_{k+1} =& \argmin_X \frac{\beta}{2} \|P_{k+1} - X + \Gamma_k\|_{\mathrm{F}}^2, \hbox{ s.t. } X^T X = I_p, \\
\Lambda_{k+1} =& \Lambda_k + P_{k+1} - Q_{k+1}, \nonumber \\
\Gamma_{k+1} =& \Gamma_k + P_{k+1} - X_{k+1}, \nonumber
\end{align}
where $\beta$ is a constant. Computing $P_{k+1}$ in~\eqref{SOCst1} requires to solve a linear system $(\beta I_n - A^T A) X = B$ for a given matrix $B$. when $m < n$ and $\beta I_n - A^T A$ is invertible (which holds in our experiments), it is solved by $X = (\beta I_n - A^T A)^{-1} B = \frac{1}{\beta} \left( B + A^T (\beta - A A^T)^{-1} A B \right)$. The parameter $\beta$ is set to be 2. The SOC method stops when $F(X_k) < F_r + 10^{-7}$, where $F_r$ is maximum of the function values given by ManPG-Ada, ManPG-Ada-D, AManPG, and AManPG-D. The SOC method has been tested in the work of Chen \etal~\cite{CMSZ2019} and it is shown therein that it is the most efficient method among the tested Euclidean space based methods.

}

Tables~\ref{table3}, ~\ref{table4} and~\ref{table5} show the performance of the five algorithms with various values of $\lambda$. In the tables, the numbers of iterations, runtime in seconds, final function values, the norms of $\|\eta_{z_k}\|_{P}$, sparsity levels and the adjusted variances~\cite{ZoHaTi2006a} are reported. The sparsity level is the portion of entries that are less than $10^{-5}$ in magnitude. The variance in the table refers to the normalized value given by the variance of the sparse PCA solution divided by the maximum variance achieved by the PCA.

\begin{table}[ht!]
\begin{center}
\caption{An average result of 20 random runs for the random data: $\whfirrev{}{p} = 4$, $n = 3000$ and $m = 40$. The subscript $k$ indicates a scale of $10^{k}$.}
\label{table3}
\vspace{0.25cm}
\makegapedcells
\setcellgapes{3pt}
\begin{tabular}{c|c|cccccc}
  \hline
  $\lambda$ & Algo & iter & time & $f$ & $\|\eta_{z_k}\|$ & sparsity & variance \\
  \hline
2.0 & SOC         & 1894 & 1.06 & $-7.02_{1}$ & $\backslash$ & 0.52 & 0.84 \\
2.0 & ManPG-Ada   & 359 & 0.35 & $-7.02_{1}$ & $5.12_{-4}$ & 0.52 & 0.84 \\
2.0 & ManPG-Ada-D & 335 & 0.37 & $-7.02_{1}$ & $5.22_{-4}$ & 0.52 & 0.84 \\
2.0 & AManPG      & 128 & 0.20 & $-7.02_{1}$ & $4.35_{-4}$ & 0.52 & 0.84 \\
2.0 & AManPG-D    & 118 & 0.21 & $-7.02_{1}$ & $4.23_{-4}$ & 0.52 & 0.84 \\
  \hline
2.5 & SOC         & 2515 & 1.43 & $-1.44_{1}$ & $\backslash$ & 0.66 & 0.72 \\
2.5 & ManPG-Ada   & 358 & 0.36 & $-1.44_{1}$ & $5.89_{-4}$ & 0.66 & 0.72 \\
2.5 & ManPG-Ada-D & 327 & 0.39 & $-1.44_{1}$ & $5.84_{-4}$ & 0.66 & 0.72 \\
2.5 & AManPG      & 130 & 0.22 & $-1.44_{1}$ & $5.13_{-4}$ & 0.66 & 0.72 \\
2.5 & AManPG-D    & 115 & 0.22 & $-1.44_{1}$ & $4.99_{-4}$ & 0.66 & 0.72 \\
  \hline
3.0 & SOC         & 3099 & 1.77 & $2.84_{1}$ & $\backslash$ & 0.83 & 0.48 \\
3.0 & ManPG-Ada   & 389 & 0.43 & $2.84_{1}$ & $6.89_{-4}$ & 0.83 & 0.48 \\
3.0 & ManPG-Ada-D & 310 & 0.43 & $2.81_{1}$ & $6.96_{-4}$ & 0.83 & 0.48 \\
3.0 & AManPG      & 166 & 0.33 & $2.80_{1}$ & $6.04_{-4}$ & 0.83 & 0.47 \\
3.0 & AManPG-D    & 134 & 0.31 & $2.73_{1}$ & $5.65_{-4}$ & 0.84 & 0.46 \\
  \hline
\end{tabular}
\end{center}
\end{table}

\begin{table}[ht!]
\begin{center}
\caption{The result for the DNA methylation data: $\whfirrev{}{p} = 4$, $n = 24589$ and $m = 113$. The subscript $k$ indicates a scale of $10^{k}$.}
\label{table4}
\vspace{0.25cm}
\makegapedcells
\setcellgapes{3pt}
\begin{tabular}{c|c|cccccc}
  \hline
  $\lambda$ & Algo & iter & time & $f$ & $\|\eta_{z_k}\|$ & sparsity & variance \\
  \hline
2.0 & SOC         & 5413 & 134.94 & $-9.43_{3}$ & $\backslash$ & 0.10 & 0.98 \\
2.0 & ManPG-Ada   & 1532 & 6.43 & $-9.43_{3}$ & $1.09_{-4}$ & 0.11 & 0.98 \\
2.0 & ManPG-Ada-D & 146 & 0.87 & $-9.43_{3}$ & $3.24_{-4}$ & 0.10 & 0.98 \\
2.0 & AManPG      & 101 & 0.81 & $-9.43_{3}$ & $1.11_{-4}$ & 0.10 & 0.98 \\
2.0 & AManPG-D    & 66 & 0.73 & $-9.43_{3}$ & $2.61_{-4}$ & 0.10 & 0.98 \\
  \hline
6.0 & SOC         & 2000 & 51.90 & $-7.74_{3}$ & $\backslash$ & 0.29 & 0.96 \\
6.0 & ManPG-Ada   & 431 & 2.66 & $-7.74_{3}$ & $3.08_{-4}$ & 0.29 & 0.96 \\
6.0 & ManPG-Ada-D & 180 & 1.57 & $-7.74_{3}$ & $8.14_{-4}$ & 0.29 & 0.96 \\
6.0 & AManPG      & 106 & 1.45 & $-7.74_{3}$ & $2.68_{-4}$ & 0.29 & 0.96 \\
6.0 & AManPG-D    & 56 & 1.13 & $-7.74_{3}$ & $5.13_{-4}$ & 0.29 & 0.96 \\
  \hline
10.0 & SOC         & 1516 & 38.01 & $-6.21_{3}$ & $\backslash$ & 0.43 & 0.94 \\
10.0 & ManPG-Ada   & 144 & 1.36 & $-6.21_{3}$ & $4.58_{-4}$ & 0.43 & 0.93 \\
10.0 & ManPG-Ada-D & 50 & 0.90 & $-6.21_{3}$ & $1.20_{-3}$ & 0.43 & 0.93 \\
10.0 & AManPG      & 66 & 1.37 & $-6.21_{3}$ & $2.02_{-4}$ & 0.43 & 0.94 \\
10.0 & AManPG-D    & 41 & 1.30 & $-6.21_{3}$ & $8.11_{-4}$ & 0.43 & 0.93 \\
  \hline
\end{tabular}
\end{center}
\end{table}

\begin{table}[ht!]
\begin{center}
\caption{An average result of 20 random runs for the synthetic data: $\whfirrev{}{p} = 5$, $n = 4000$ and $m = 400$. The subscript $k$ indicates a scale of $10^{k}$.}
\label{table5}
\vspace{0.25cm}
\makegapedcells
\setcellgapes{3pt}
\begin{tabular}{c|c|cccccc}
  \hline
  $\lambda$ & Algo & iter & time & $f$ & $\|\eta_{z_k}\|$ & sparsity & variance \\
  \hline
1.0 & SOC         & 529 & 9.05 & $-3.64_{2}$ & $\backslash$ & 0.61 & 0.95 \\
1.0 & ManPG-Ada   & 41 & 0.11 & $-3.64_{2}$ & $3.17_{-4}$ & 0.61 & 0.95 \\
1.0 & ManPG-Ada-D &  24 & 0.08 & $-3.64_{2}$ & $3.97_{-4}$ & 0.61 & 0.95 \\
1.0 & AManPG      &  41 & 0.16 & $-3.64_{2}$ & $2.09_{-4}$ & 0.61 & 0.95 \\
1.0 & AManPG-D    &  31 & 0.14 & $-3.64_{2}$ & $2.64_{-4}$ & 0.61 & 0.95 \\
  \hline
1.5 & SOC         & 412 & 7.29 & $-2.99_{2}$ & $\backslash$ & 0.74 & 0.93 \\
1.5 & ManPG-Ada   & 37 & 0.10 & $-2.99_{2}$ & $4.51_{-4}$ & 0.74 & 0.93 \\
1.5 & ManPG-Ada-D & 19 & 0.08 & $-2.99_{2}$ & $4.44_{-4}$ & 0.74 & 0.93 \\
1.5 & AManPG      & 33 & 0.15 & $-2.99_{2}$ & $2.94_{-4}$ & 0.74 & 0.93 \\
1.5 & AManPG-D    & 25 & 0.13 & $-2.99_{2}$ & $3.59_{-4}$ & 0.74 & 0.93 \\
  \hline
2.0 & SOC         & 375 & 6.36 & $-2.39_{2}$ & $\backslash$ & 0.80 & 0.91 \\
2.0 & ManPG-Ada   & 46 & 0.11 & $-2.39_{2}$ & $5.85_{-4}$ & 0.80 & 0.91 \\
2.0 & ManPG-Ada-D & 17 & 0.07 & $-2.39_{2}$ & $6.37_{-4}$ & 0.80 & 0.91 \\
2.0 & AManPG      & 33 & 0.14 & $-2.39_{2}$ & $3.82_{-4}$ & 0.80 & 0.91 \\
2.0 & AManPG-D    & 23 & 0.12 & $-2.39_{2}$ & $3.92_{-4}$ & 0.80 & 0.91 \\
  \hline
\end{tabular}
\end{center}
\end{table}

\whcommtrd{}{It can be seen from the tables that the SOC method takes the most computational time to achieve a similar accuracy.}
In addition, the tables show that AManPG shares the same fast convergence as the Euclidean FISTA method in terms of the number of iterations.
Note that the additional computations on the safeguard, the  retraction, as well as the inverse of retraction make the  per iteration cost of AManPG higher than that of ManPG-Ada. 
Despite this, due to the significant reduction on the number of iterations, AManPG is still substantially faster than ManPG-Ada in terms of the computational time for the random data and the real DNA data (see Tables~\ref{table3} and \ref{table4}).  %In Table~\ref{table5}, though AManPG is not faster than ManPG and ManPG-Ada, it is still a competitive method. The reason is that the optimization problem using the synthetic data is easy in the sense that the numbers of iterations required by all algorithms are small. Therefore, the reduction of iterations in AManPG does not make up the extra cost in its iterations.
For the synthetic data, Table~\ref{table5} suggests this problem is relatively easier in the sense that all the algorithms are able to achieve the convergence within a small number of iterations. Thus, the two AManPG algorithms do not exhibit the 
the computational advantage in terms of the runtime due to the additional costs in each iteration.
Moreover, it is   evident  that using the diagonal weight significantly improves the efficiency of ManPG, ManPG-Ada, and AManPG both in terms of the number of iterations and in terms of the computational time.
%In addition,  the  {\em function values versus iterations} plots are presented in Figure~\ref{RPG:SPCAOB_f}, which visually shows the accelerated behavior of AManPG and the effect of the diagonal weight.

\subsection{Efficiency for large-scale problems}

\whfirrev{}{
In this section, the efficiency of the representative method, AManPG-D, is shown for multiple values of $n, m, p$. The values of $\lambda$ are tuned such that the solutions have reasonable sparsities. 
As shown in Table~\ref{table6}, the trend of the computational time roughly follows the complexity analysis discussed at the end of Section~\ref{sec:SSN}. Moreover, AManPG-D exhibits high efficiency for the sparse PCA model~\eqref{eq:sparsePCA} in the sense that it is able to solve the problem with $n = 96000, m = 1280, p = 16$ within half a minute.
}

\begin{table}[ht!]
\begin{center}
\caption{Performance of AManPG-D with multiple values of $n$, $m, p, \lambda$. An average result of 10 random runs for the random data. The subscript $k$ indicates a scale of $10^{k}$.}
\label{table6}
\vspace{0.25cm}
\makegapedcells
\setcellgapes{3pt}
\begin{tabular}{c|cccccccc}
  \hline
  $n$ & 3000 & 6000 & 6000 & 6000 & 12000 & 24000 & 48000 & 96000 \\
  $m$ & 40 & 40 & 80 & 80 & 160 & 320 & 640 & 1280 \\
  $p$ & 4 & 4 & 4 & 8 & 8 & 8 & 16 & 16 \\
  $\lambda$ & 3 & 3 & 2 & 2 & 1.5 & 1 & 0.75 & 0.5 \\
  \hline
iter              & 121 & 123 & 139 & 189 & 175 & 146 & 98 & 55 \\                                                                  
time              & 0.25 & 0.37 & 0.46 & 1.34 & 2.72 & 4.40 & 18.97 & 29.32 \\                                                      
$f$               & $2.46_{1}$ & $-7.83_{1}$ & $1.06_{1}$ & $2.70_{1}$ & $4.38_{1}$ & $1.52_{1}$ & $8.02_{1}$ & $2.71_{1}$ \\       
$\|\eta_{z_k}\|$  & $5.42_{-4}$ & $7.65_{-4}$ & $8.07_{-4}$ & $1.75_{-3}$ & $2.65_{-3}$ & $3.88_{-3}$ & $1.17_{-2}$ & $1.55_{-2}$ \\
sparsity          & 0.85 & 0.57 & 0.76 & 0.78 & 0.84 & 0.77 & 0.84 & 0.76 \\                                                        
variance          & 0.43 & 0.80 & 0.60 & 0.58 & 0.47 & 0.59 & 0.48 & 0.60 \\ 
  \hline
\end{tabular}
\end{center}
\end{table}

\whcommsec{[Figure~\ref{figure4} is interesting but may be not important in this paper. You can remove it if you want. ]}{}  \kwcomm{[Prefer to keep this]}{}

\subsection{Compare with sparse PCA models in GPower}

For the synthetic data, because there exists a ground truth, it is favorable to present the principal components returned by the PCA and that returned by the Riemannian proximal methods for the sparse PCA formulation, see Figure~\ref{figure4} (only the principal components obtained from AManPG-D are reported as a representative). The figure clearly shows that the latter one is  more likely to capture the sparse structure of the loading vector.
%The proposed method ARPG only needs approximately one tenth iterations and one fourth computational times as those needed in ManPG to obtain similar accurate solutions.

We have also compared AManPG-D and the GPower with $l_1$ and $l_0$ norms (designed for a
different sparse PCA model, see the work of Journ\'ee \cite{JoNeRiSe2010a}) with various sparsity levels. The results are presented in Figures~\ref{figure1}, \ref{figure2} and \ref{figure3} for the random data, the real DNA data and the synthetic data, respectively. We can see that AManPG-D for \eqref{eq:sparsePCA} produces
 an orthonormal loading matrix while does not lose much variance compared with GPower.

%\begin{figure}[ht!]
%\includegraphics[width=1.0\textwidth]{FunVsIter}
%\caption{
%Plots of {\em function values versus iterations} for three typical instances. Left: Random data, $r=4$, $\lambda=2.5$; Middle: DNA methylation data, $r=4$, $\lambda=6$; Right: Synthetic data, $r = 5$, $\lambda = 1.5$.
%}
%\label{RPG:SPCAOB_f}
%\end{figure}

\begin{figure}[ht!]
%\begin{center}
%\includegraphics[width=1\textwidth]{ARM_versus_k.eps}\\
\hspace{-0.05\textwidth}\includegraphics[width=1.1\textwidth]{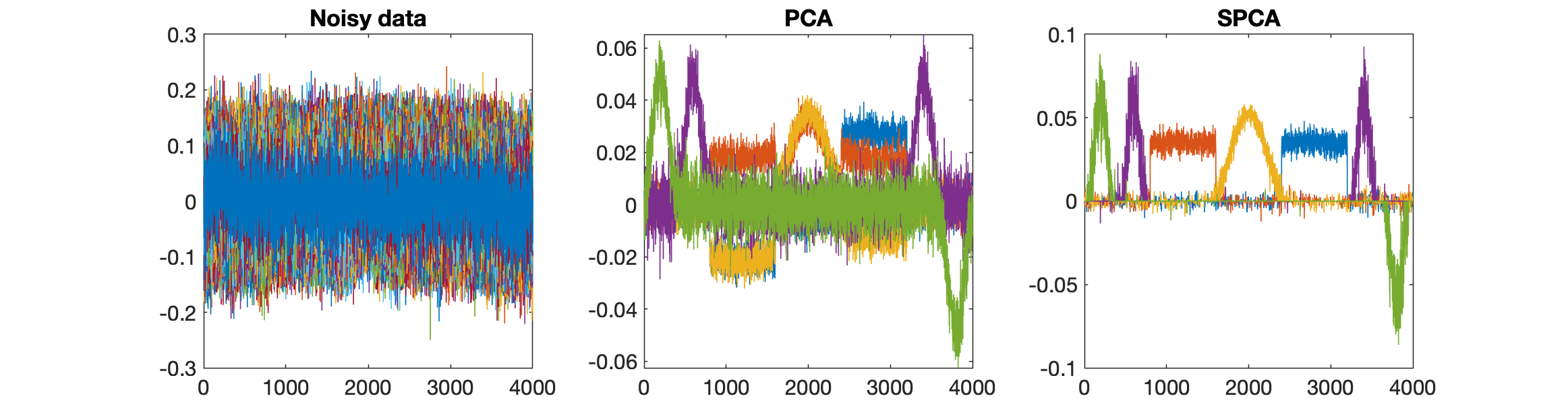}
\vspace{-1em}
\caption{Comparison between the principal components from the PCA and that from the Sparse PCA by AManPG-D with $\lambda = 1.5$.
}%corresponding to noisy measurements.% \ph{Top: Remove the top figure.  It is redundant with the terminal values of Figure 3. }{}}
\label{figure4}
%\end{center}
\end{figure}
%{We also compare the solution given by AManPG to the one given by GPower \cite{JoNeRiSe2010a}}. Figures~\ref{figure1} and~\ref{figure2} display the results for AManPG as well  GPower with $l_1$ and $l_0$ norms (designed for a different sparse PCA model) on various sparsity levels. We can see that AManPG produces an orthonormal loading matrix while does not lose much  variance  compared with GPower.

\begin{figure}[ht!]
%\begin{center}
%\includegraphics[width=1\textwidth]{ARM_versus_k.eps}\\
\hspace{-0.05\textwidth}\includegraphics[width=1.1\textwidth]{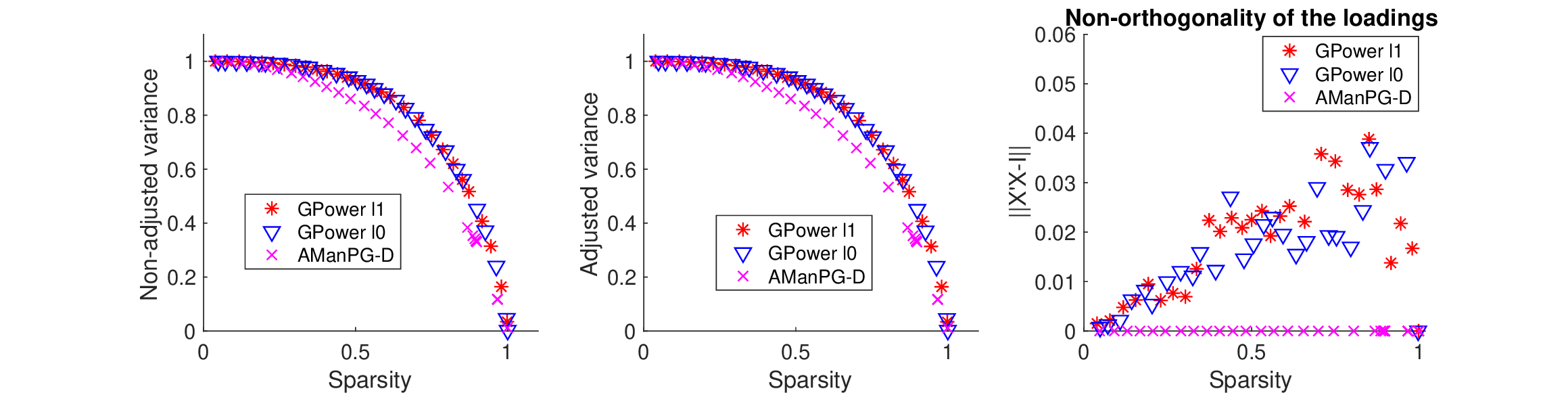}
\vspace{-1em}
\caption{Sparse PCA by AManPG-D and sparse PCA by GPower. Matrix $A \in \mathbb{R}^{3000 \times 40}$ is generated randomly. The number of components $\whfirrev{}{p}$ is set to be $4$.
}%corresponding to noisy measurements.% \ph{Top: Remove the top figure.  It is redundant with the terminal values of Figure 3. }{}}
\label{figure1}
%\end{center}
\end{figure}

\begin{figure}[ht!]
%\begin{center}
%\includegraphics[width=1\textwidth]{ARM_versus_k.eps}\\
\hspace{-0.05\textwidth}\includegraphics[width=1.1\textwidth]{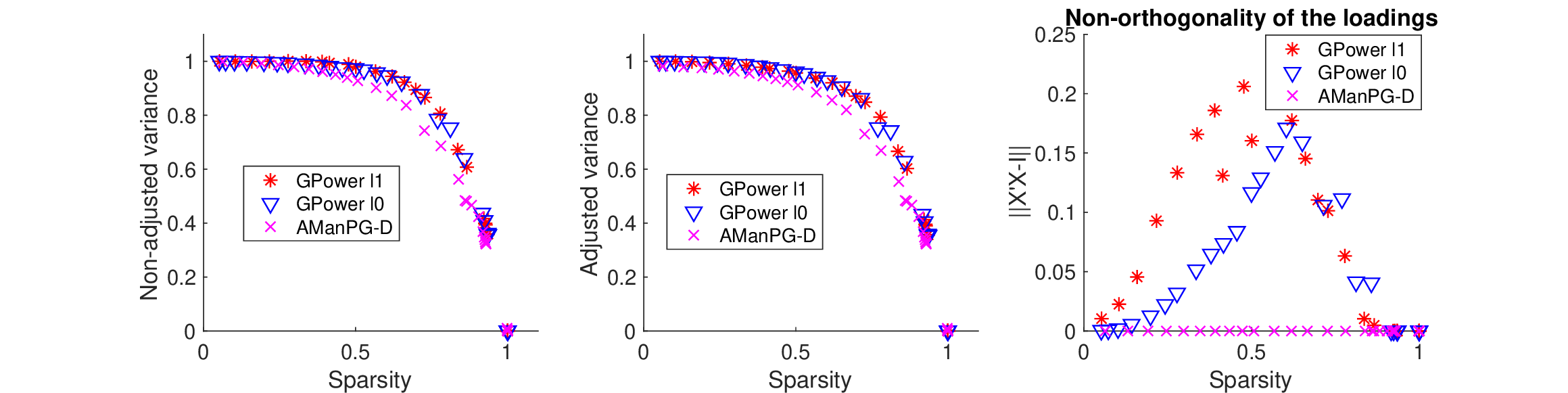}
\vspace{-1em}
\caption{Sparse PCA by AManPG-D and sparse PCA by GPower. Matrix $A \in \mathbb{R}^{24589 \times 113}$ is from the DNA methylation data. The number of components $\whfirrev{}{p}$ is set to be $4$.
}%corresponding to noisy measurements.% \ph{Top: Remove the top figure.  It is redundant with the terminal values of Figure 3. }{}}
\label{figure2}
%\end{center}
\end{figure}

\begin{figure}[ht!]
%\begin{center}
%\includegraphics[width=1\textwidth]{ARM_versus_k.eps}\\
\hspace{-0.05\textwidth}\includegraphics[width=1.1\textwidth]{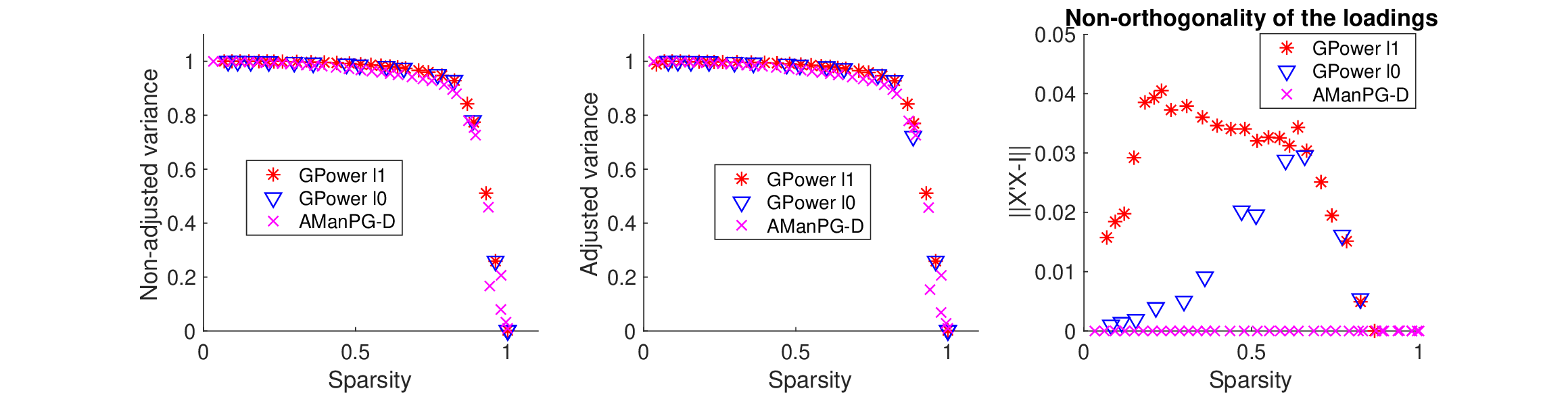}
\vspace{-1em}
\caption{Sparse PCA by AManPG-D and sparse PCA by GPower. Matrix $A \in \mathbb{R}^{4000 \times 400}$ is from the synthetic data. The number of components $\whfirrev{}{p}$ is set to be $5$.
}%corresponding to noisy measurements.% \ph{Top: Remove the top figure.  It is redundant with the terminal values of Figure 3. }{}}
\label{figure3}
%\end{center}
\end{figure}

\whcommsec{[Conjuncture: In Figures~\ref{figure1},~\ref{figure2}, and~\ref{figure3}, I think the fact that increasing the sparsity reduces the variance less significantly implies that the extracting the principal components is easier. Therefore, probably the optimization is also easier. ]}{} \kwcomm{[Have done nothing about this point]}{} \whcommsec{[Just a remark for myself. No need to do anything.]}{}

%%%%%%

\section{Conclusion and future directions}\label{sec:con}
%\whcommsec{[Need be modified]}{}
%This paper proposes an accelerated Riemannian proximal gradient  method for the sparse PCA model which combines the $l_1$-regularization and manifold constraint. Numerical experiments show that the accelerated algorithm is substantially  faster than the existing proximal gradient type methods, and the algorithm are  able to achieve a good balance between sparsity, orthogonality and variance explained. Stationary point convergence of the algorithm has been carefully justified.

In this paper we extend the well-known accelerated first order method FISTA from the Euclidean setting to the Riemannian setting. Moreover,  a diagonal preconditioning strategy is also presented which can further accelerate the convergence of the Riemannian proximal gradient methods. Empirical evaluations on the sparse PCA problems have established the computational advantages of the proposed methods. Stationary point convergence of the algorithm has been carefully justified.

There are several lines of research for future directions. In addition to the stationary point analysis, it is also desirable to study the local convergence rate of the Riemannian proximal methods.  It is also interesting to develop and study the high order Riemannian methods for the nonsmooth Riemannian optimization problems with the splitting structure. For example,  in this paper we only use the diagonal weight to accelerate the convergence of the algorithms for the computational efficiency of the Riemannian proximal subproblem. It is very natural to further consider the Newton type method for this kind of problems. In this case the crux would be to develop efficient algorithms for the scaled proximal mapping on the tangent space. In the work of Li \etal \cite{LiSunToh18} a highly efficient semi-smooth Newton augmented Lagrangian method is proposed for the Lasso problem. Due to the similar structures between the Lasso problem and the sparse PCA problem, it is intriguing to see whether or not the method can be extended to solve the sparse PCA problem.
%%%%%%%%

%\section*{Data Availability Statement}
%
%The data that support the findings of this study are openly available in NCBI webpage at \url{https://www.ncbi.nlm.nih.gov/}, reference number GSE32393.

\section*{acknowledgements}
We would like to thank Shiqian Ma for kindly sharing their codes with us, and thank Xudong Li for the helpful discussion regarding to the semi-smooth Newton method. This study does not have any conflicts to disclose.

\bibliographystyle{alpha}
\bibliography{WHlibrary,ref}

\appendix

\section{Proof of Lemma~\ref{le2}} \label{app1}

\begin{proof}
We first prove the result of the work~\cite[Lemma~5.1]{CMSZ2019}. Note that the proof is slightly different due to the presence of the weight matrix $W$.

Let $\ell_{\tilde{z}_i}(\xi) = \inner[]{ \grad f( \tilde{z}_i ) }{\xi_{\tilde{z}_i}} + \frac{1}{2 \mu} \|\xi\|_{W_{\tilde{z}_i}}^2 + g( \tilde{z}_i + \xi )$. 
Define functions $\tilde{\ell}_{\tilde{z}_i} (\xi) = \langle W_{\tilde{z}_i}^{-1/2} \xi_{\tilde{z}_i}, \xi \rangle+\frac{1}{2\mu}\| \xi \|_F^2 + g(x+ W_{\tilde{z}_i}^{-1/2} \xi)$ and $\theta(\xi) = W_{\tilde{z}_i}^{1/2} \xi$. We have $\ell_{\tilde{z}_i} = \tilde{\ell}_{\tilde{z}_i} \circ \theta$.
By $\frac{1}{\mu}$-strongly convexity of $\tilde{\ell}_{\tilde{z}_i}$, we have
\[
\tilde{\ell}_{\tilde{z}_i} ( \hat{\xi} ) \geq \tilde{\ell}_{\tilde{z}_i} ( {\xi} ) + \inner[]{ \partial \tilde{\ell}_{\tilde{z}_i} ( {\xi} ) }{ \hat{\xi} - \xi } + \frac{1}{2 \mu} \|\hat{\xi} - \xi\|_F^2, \quad \forall \hat{\xi}, \xi \in \mathbb{R}^{n \times p},
\]
which together with the full rank of $W_{\tilde{z}_i}$ yields
\begin{equation} \label{e20}
\ell_{\tilde{z}_i}(\hat{\xi}) \geq \ell_{\tilde{z}_i}(\xi) + \inner[]{\partial \ell_{\tilde{z}_i}( \xi )}{ \hat{\xi} - \xi } + \frac{1}{2 \mu} \| \hat{\xi} - \xi \|_{W_{\tilde{z}_i}}^2 \quad \forall \hat{\xi}, \xi \in \mathbb{R}^{n \times p}.
\end{equation}
By the definition of $\eta_{\tilde{z}_i}$ in Step~\ref{alg:Safeguard:st1} of Algorithm~\ref{alg:Safeguard} and the optimality condition, we have $0 \in \P_{\T_{\tilde{z}_i}} \partial \ell_{\tilde{z}_i} (\eta_{\tilde{z}_i})$. It follows from~\eqref{e20} that
\[
\ell_{\tilde{z}_i}( 0 ) \geq \ell_{\tilde{z}_i}(\eta_{\tilde{z}_i}) + \frac{1}{2 \mu} \| \eta_{\tilde{z}_i} \|_{W_{\tilde{z}_i}}^2,
\]
which implies
\[
g( \tilde{z}_{i} ) \geq \inner[]{ \grad f( \tilde{z}_i ) }{ \eta_{\tilde{z}_i} } + \frac{1}{2 \mu} \| \eta_{\tilde{z}_i} \|_{W_{\tilde{z}_i}}^2 + g( \tilde{z}_i + \eta_{\tilde{z}_i} ) + \frac{1}{2 \mu} \| \eta_{\tilde{z}_i} \|_{W_{\tilde{z}_i}}^2.
\]
By the convexity of $g$, we have
\[
g( \tilde{z}_{i} + \alpha \eta_{\tilde{z}_i} ) - g( \tilde{z}_{i} ) = g( \alpha (\tilde{z}_{i} + \eta_{\tilde{z}_i}) + (1-\alpha) \tilde{z}_{i} ) - g(\tilde{z}_{i}) \leq \alpha ( g( \tilde{z}_{i} + \eta_{\tilde{z}_i} ) - g( \tilde{z}_{i} ) ) \quad \forall \alpha \in [0, 1].
\]
Combining the above two inequalities yields
\begin{equation}\label{e21}
\ell_{\tilde{z}_i}(\alpha \eta_{\tilde{z}_i}) - \ell_{\tilde{z}_i}( 0 ) \leq \frac{\alpha(\alpha - 2)}{2 \mu} \| \eta_{\tilde{z}_i} \|_{W_{\tilde{z}_i}}^2 \leq \frac{\alpha(\alpha - 2) {\kappa}}{2 \mu} \| \eta_{\tilde{z}_i} \|_{F}^2,
\end{equation}
where the last inequality follows from Assumption~\ref{as4}.

Using~\eqref{e21}, the first item of Lemma~\ref{le2} can be obtained by exactly following the steps in the work~\cite[Lemma~5.2]{CMSZ2019}.

By the definition of $\eta_{\tilde{z}_i}$ in Step~\ref{alg:Safeguard:st1} of Algorithm~\ref{alg:Safeguard} and the first order optimality condition, we have
\begin{equation} \label{e22}
0 \in \frac{1}{\mu} W_{\tilde{z}_i} \eta_{\tilde{z}_i} + \grad f(\tilde{z}_i) + \P_{\T_{\tilde{z}_i} \mathcal{M}} \partial g(\tilde{z}_i + \eta_{\tilde{z}_i}).
\end{equation}
If $\eta_{\tilde{z}_i} = 0$, then it follows from~\eqref{e22} that
\[
0 \in \grad f(\tilde{z}_i) + \P_{\T_{\tilde{z}_i} \mathcal{M}} \partial g(\tilde{z}_i + \eta_{\tilde{z}_i}),
\]
which is exactly the first order optimality condition of Problem~\eqref{eq:optimization}. This completes the proof for the second item of Lemma~\ref{le2}.
\end{proof}

\end{document}